\documentclass[reqno]{conm-p-l}

\pdfoutput=1  %for arXiv
\usepackage{general,strings,txfonts,ifthen}

\usepackage[bookmarks, colorlinks=true, pdfstartview=FitV, linkcolor=blue, citecolor=blue, urlcolor=blue]{hyperref}
\usepackage{cite}  % <= put after hyperref

\newcommand{\bd}{\mathrm{bd}\,} %boundary of a set
\newcommand{\mydot}{\,\cdot\,}

\renewcommand{\linenopax}{}%{\par\vspace{-2ex}}  %fix for bug in package lineno.sty
\renewcommand{\ii}{\mathbbm{i}}  %imaginary unit
\newcommand{\pole}{\ensuremath{\gw}\xspace}   %pole / complex dimension
\newcommand{\simtset}{\{\simt_1,\ldots,\simt_N\}}
\newcommand{\scale}{\ensuremath{\ell}\xspace}     %scaling ratio
\newcommand{\abscissa}{\ensuremath{D}\xspace}  %abscissa of convergence
\newcommand{\Words}{\sW}     %set of finite words
\newcommand{\spam}{\gG}
\newcommand{\Mink}{\sM}
\newcommand{\Minkl}{\ensuremath{\Mink_\star}\xspace}
\newcommand{\Minku}{\ensuremath{\Mink^\star}\xspace}
\newcommand{\Minka}{\ensuremath{\cj{\Mink}}\xspace}
\newcommand{\polestrip}{\sP\xspace}
\newcommand{\SG}{\protect{\ensuremath{\sS\negsp[3]\raisebox{-1pt}{\sG}\negsp[3]}}\xspace}     %Sierpinksi gasket
\newcommand{\bM}{\ensuremath{\sM^{\operatorname{out}}}\xspace} %generalized generator
\newcommand{\Sf}{\ensuremath{S_{\negsp[7]f}}\xspace} %screen
\newcommand{\Wf}{\ensuremath{W_{\negsp[6]f}}\xspace} %screen

  \let\oldmarginpar\marginpar
  \renewcommand\marginpar[1]
    {\-\oldmarginpar[\raggedleft\scriptsize #1]%
   {\raggedright \scriptsize #1}}

\numberwithin{equation}{section}
\numberwithin{theorem}{section}
\numberwithin{figure}{section}

\begin{document}
  %\linenumbers

  \title[Minkowski measurability of self-similar tilings and fractals (monophase case)]
    {Minkowski measurability results for self-similar tilings and fractals with monophase generators}

  \author{Michel L. Lapidus}
  \address{University of California, Department of Mathematics, Riverside, CA 92521-0135 USA}
  \email{\href{mailto:lapidus@math.ucr.edu}{lapidus@math.ucr.edu}}

  \author{Erin P. J. Pearse}
  \address{California Polytechnic State University, Department of Mathematics, San Luis Obispo, CA 93407-0403 USA}
  \email{\href{mailto:epearse@calpoly.edu}{epearse@calpoly.edu}}

  \author{Steffen Winter}
  \address{Karlsruhe Institute of Technology, Department of Mathematics, 76128 Karlsruhe, Germany}
  \email{\href{mailto:steffen.winter@kit.edu}{steffen.winter@kit.edu}}

  \begin{abstract}
    In a previous paper \cite{Pointwise}, the authors obtained tube formulas for certain fractals under rather general conditions. Based on these formulas, we give here a characterization of Minkowski measurability of a certain class of self-similar tilings and self-similar sets. Under appropriate hypotheses, self-similar tilings with simple generators (more precisely, monophase generators) are shown to be Minkowski measurable if and only if the associated scaling zeta function is of nonlattice type. Under a natural geometric condition on the tiling, the result is transferred to the associated self-similar set (i.e., the fractal itself). Also, the latter is shown to be Minkowski measurable if and only if the associated scaling zeta function is of nonlattice type.
  \end{abstract}

  \date{\textbf{\today}}
  \thanks{\textbf{\today.} The work of MLL was partially supported by the US National Science Foundation under the research grant DMS-0707524, as well as by the Institut des Hautes \'Etudes Scientifiques (IH\'ES), in Bures-sur-Yvette, France, where he was a visiting professor while this paper was completed. }

\maketitle

%%!TEX root = mink-meas.tex

\section{Introduction} \label{sec:intro}
Let $A$ be a bounded subset in Euclidean space $\bR^d$. For $0\le \alpha\le d$, we denote by
\linenopax
\begin{align} \label{eqn:M-content}
  \sM_\alpha(A):=\lim_{\ge\to 0^+}\frac{\lambda_d(A_\ge)}{\ge^{d-\alpha}}
\end{align}
the \emph{$\alpha$-dimensional Minkowski content} of $A$ whenever this limit exists (as a value in $[0,\infty]$). Here $\lambda_d$ denotes the Lebesgue measure in $\bR^d$ and 
\linenopax
\begin{align}\label{eqn:A_ge}
  A_\ge:=\{x\in \bR^d: d(x,A)\leq \ge\}
\end{align}
 is the \emph{$\ge$-parallel set} of $A$ (where $d(x,A):=\inf\{\|x-a\|:a\in A\}$ is the Euclidean distance of $x$ to the set $A$). $A$ is called \emph{Minkowski measurable (of dimension $\alpha$)}, if $\sM_\alpha(A)$ exists and satisfies $0<\sM_\alpha(A)<\infty$.
The question whether a set $A$ is Minkowski measurable of some dimension $\alpha$ has received considerable attention in the past. One motivation for studying this notion is the suggestion by Mandelbrot in \cite{Mandelbrot95}, to use it as a characteristic for the texture of sets \cite[\S{}X]{Mandelbrot82}. Mandelbrot called the number $1/\sM_\alpha(A)$ the \emph{lacunarity} of a set $A$ and observed that for sets in $\bR$ small lacunarity corresponds to spatial homogeneity of the set, i.e.\ small, uniformly distributed holes, while large lacunarity corresponds to clustering and large holes between different clusters; see also \cite{BedFi, Frantz, FGNT1} and \cite[\S12.1.3]{FGCD}. The notion of Minkowski content attracted even more attention in connection with the (modified) Weyl--Berry conjecture (as formulated in \cite{Lap:FD})%
\footnote{We refer to Berry's papers \cite{Berr1,Berr2} for the original Weyl--Berry conjecture and its physical applications. For early mathematical work on this conjecture and its modifications, see \cite{BroCa, Lap:FD, Lap:Dundee, LaPo1, FleVa,  LaPo2}, for example. See also \cite[\S12.5]{FGCD} for a more extensive list of later work.}
, proved for subsets of $\bR$ in 1993 by Lapidus and Pomerance \cite{LaPo1}. %(and disproved in any higher dimension in~\cite{LaPo2}).
It establishes a relation between the spectral asymptotics of the Laplacian on a bounded open set and the Minkowski content of its boundary. A key step towards this result is the characterization of Minkowski measurability of compact subsets of \bR (or equivalently, of fractal strings) obtained in \cite{LaPo1} (and given a new proof in \cite{Falconer95} and more recently in \cite{RatajWinter2}). In particular, this led to a reformulation of the Riemann hypothesis in terms of an inverse spectral problem for fractal strings; see \cite{LaMa}.

In one dimension, the Minkowski content of a set is completely determined by the sequence of the lengths of its complementary intervals; cf.~\cite{LaPo1} or \cite{Falconer95}. In particular, the geometric arrangement of the intervals is irrelevant, in sharp contrast to the situation for the Hausdorff measure. Such sequences of lengths are nowadays known as \emph{fractal strings} and have become an independent object of study with numerous applications, e.g.\ in spectral geometry and number theory; see \cite{FGCD, FGNT1} and the references therein. In particular, they allowed the introduction and the development of a rigourous theory of complex dimensions.

One recent focus of research are generalizations of the theory to higher dimensions.
A natural analogue to a fractal string, which can also be viewed as a collection of disjoint open intervals (or as a collection of scaled copies of a generating interval), are the \emph{fractal sprays} introduced in \cite{LaPo2}.
A fractal spray $\sT=\{T_i\}$ is a collection of pairwise disjoint scaled copies $T_i$, $i\in \bN$ of a bounded open set $G$ in $\bR^d$. The associated scaling ratios -- arranged in nonincreasing order -- form a fractal string. The set $G$ is referred to as the \emph{generator} of $\sT$. Fractal sprays naturally arise in connection with iterated function systems. A tiling of the convex hull (or, more generally, of some feasible open set from the open set condition) of a self-similar set was constructed in \cite{Pe2, SST, GeometryOfSST}. This tiling consists of a countable collection of scaled copies of some generator and is thus a fractal spray. Such self-similar tilings can be used to decompose the $\ge$-parallel set of a self-similar set $F$ and to derive in this way a tube formula for $F$, i.e.\ a formula describing the volume $V(F_\ge)$ of the parallel sets $F_\ge$ as a function of the parallel radius $\ge$; we refer to Section~\ref{sec:terms} for more details. An essential step towards such tube formulas for self-similar sets are tube formulas for self-similar tilings (or, more generally, fractal sprays) $\sT$, which describe the inner parallel volume $V(T,\ge)$ of the union set $T:=\bigcup_i T_i$, that is the volume
\linenopax
\begin{align}\label{eqn:tiling-tube}
  V(T,\ge) := \gl_d\left(\{x \in T \suth \dist(x, T^c) \leq \ge\}\right),
  \qq \ge \geq 0,
 \end{align}
as a sum of residues of some associated zeta function which is a generating function of the geometry of the tiling.
Tube formulas for fractal sprays have first been obtained in \cite{Pe2, TFSST, TFCD, Pointwise}, generalizing the tube formulas for fractal strings in \cite{FGCD, FGNT1} to higher dimensions. The topic has been pursued in \cite{Demir, Demir2, Kocak1, Kocak2, DistanceZeta, Box-countingZeta}.
We refer to Section~\ref{sec:Self-similar-tilings} for more details; see also Remark~\ref{rem:fractal-boundaries} in particular.

One particular application of such formulas is the characterization of Minkowski measurability. In one dimension it is well known that a self-similar set is Minkowski measurable if and only if it is nonlattice; see \cite{FGCD} and \cite[\S8.4]{FGNT1}. (See also \cite{LaPo1, Lap:Dundee, Falconer95} for partial results, along with Remark~\ref{rem:nonlattice=nonlattice}.)
For subsets in $\bR^d$, $d\geq 2$, this is an open conjecture, see e.g.\ \cite[Conj.~3]{Lap:Dundee} and \cite[Rem.~12.19]{FGCD}. It was partially answered by Gatzouras~\cite{Gatzouras}, who proved that nonlattice self-similar sets in $\bR^d$ are Minkowski measurable. Therefore, it essentially remains to show the nonmeasurability in the lattice case, which, for subsets of $\bR$, can be proved using tube formulas (as in \cite[\S8.4.2]{FGCD}).
With these results in mind, in both \cite[Cor.~8.5]{TFCD} and \cite[Rem.~4.4 and \S8.4]{Pointwise}, the authors alluded to the fact that several results concerning Minkowski measurability follow almost immediately from the tube formulas; see also \cite[Rem.~10.6]{TFCD}. %However, no details were given.
The purpose of this paper is to supply the missing arguments for % in this appendix. For now, we consider
the special case of self-similar tilings
%(as opposed to the greater generality of \emph{fractal sprays}, as in \cite{TFCD,Pointwise})
with monophase generators and for self-similar sets possessing such tilings.
More specifically, we give precise geometric conditions, the most restrictive of them being the existence of a polynomial expansion for the inner parallel volume of the generator, under which the lattice-nonlattice dichotomy of Minkowski measurabilty carries over to higher dimensions.
The question of Minkowski measurability of self-similar sets and tilings with more general generators will be considered in \cite{Minko} using the general pointwise tube formulas derived in \cite{Pointwise}.

The paper is organized as follows.
In Section~\ref{sec:terms} we recall the construction of self-similar tilings and what it means for a generator of such a tiling to be monophase. In Section~\ref{sec:Exact-pointwise-tube-formulas}, we recall those tube formulas from \cite{Pointwise} needed for the proof of our main results on Minkowski measurability, which are then formulated and proved in Section~\ref{sec:Minkowski-measurability} for self-similar tilings and in Section~\ref{sec:fractals-themselves} for the associated self-similar sets.

\section{Self-similar tilings and their generators}
\label{sec:terms}

  All notations and notions used in the sequel are described in detail in \cite{Pointwise}; for the general theory of fractals strings and complex dimensions, we refer to \cite{FGCD}.

%In \cite{SST,GeometryOfSST,TFCD}, the focus is on self-similar tilings. Such an object is a
%We study fractal sprays associated to
Let $\simtset$, $N \geq 2$ be an iterated function system (IFS), where each $\simt_n$ is a contractive similarity mapping of \bRd with scaling ratio $r_n \in (0,1)$. For $A \ci \bRd$, we write $\simt(A) := \bigcup_{n=1}^N \simt_n(A)$. The \emph{self-similar set} $\attr$ generated by the IFS $\simtset$ is the unique compact and nonempty solution of the fixed-point equation $\attr = \simt(\attr)$\,; cf.~\cite{Hut}. The fractal \attr is also called the \emph{attractor} of $\simtset$.
%
%, i.e.,
%\linenopax
%\begin{align*}%\label{eqn:}
%  \simt_n(x) = r_n M_n(x) + t_n,
%\end{align*}
%where $M \in O(d)$ is a rotation and/or reflection, $r_n \in (0,1)$ is a contractive scaling ratio, and $t_n \in \bRd$ is a translative component (i.e. a vector).
We study the geometry of the attractor by studying the geometry of a certain tiling of its complement, which is constructed via the IFS  as follows.
%via its {tube formula}.

%\begin{defn}\label{def:tube-formula}
%  The \emph{inner parallel volume} of a bounded subset $A \ci \bRd$ is the function
%\linenopax
%\begin{align*}%\label{eqn:}
%  V(A,\ge) := \gl_d\left(\{x \in A \suth \dist(x, A^c) \leq \ge\}\right),
%  \qq \ge \geq 0,
%\end{align*}
%where $\gl_d$ is $d$-dimensional Lebesgue measure and $\dist(x, A^c) = \inf_{y \in \bRd \less A} |x-y|$ refers to the usual Euclidean distance.
%  A \emph{tube formula} is an explicit expression for $V(A,\ge)$.
%\end{defn}

%We are interested in tube formulas of \emph{self-similar tilings}.
%Such a set (or rather, collection of sets) is a partial decomposition of the complement of \attr which These tilings have been studied in some detail in \cite{SST,GeometryOfSST,TFCD, Demir, Demir2, KocakNote, Kocak1, Kocak2}. A formula for the volume of the inner tubular neighborhood of this tiling was given in \cite{Pointwise}, and we now investigate some of the implications of this formula for the Minkowski measurability of the attractor (as defined precisely in \S\ref{sec:Minkowski-measurability}).

The construction of a self-similar tiling requires the IFS to satisfy the \emph{open set condition} and a \emph{nontriviality condition}.
%The open set condition is well-studied, see , for example. The nontriviality condition and the following two propositions are excerpted from \cite{GeometryOfSST}.

\begin{defn}\label{def:OSC}
  A self-similar system $\simtset$ (or its attractor \attr) satisfies the \emph{open set condition} (OSC) if and only if there is a nonempty open set $O \ci \bRd$ such that
  \linenopax
  \begin{align}
    \simt_n(O) &\ci O, \q n=1, 2,\dots, N \label{eqn:def:OSC-containment} \\
    \simt_n(O) &\cap \simt_m(O) = \emptyset \text{ for } n \neq m.
      \label{eqn:def:OSC-disjoint}
  \end{align}
  In this case, $O$ is called a \emph{feasible open set} for $\simtset$ (or \attr); cf.~\cite{Hut, Falconer, BandtNguyenRao}.
\end{defn}

\begin{defn}\label{def:nontriv}
  A self-similar set \attr satisfying OSC is said to be \emph{nontrivial} if there exists a feasible open set $O$ such that
  \begin{equation}\label{eqn:nontriv}
    O\not \ci \simt(\cj{O})\,,
  \end{equation}
  where $\cj{O}$ denotes the closure of $O$; otherwise, \attr is called \emph{trivial}.
\end{defn}

This condition is needed to ensure that the set
  $O\setminus\simt(\cj{O})$ in Definition~\ref{def:self-similar-tiling} is nonempty.  It turns out that nontriviality is independent of the particular choice of the set $O$. It is shown in \cite{GeometryOfSST} that \attr is trivial if and only if it has interior points, which amounts to the following characterization of nontriviality:

\begin{prop}[{\cite[Cor.~5.4]{GeometryOfSST}}]
  \label{cor:OSC-dimension-d-implies-trivial}
  Let $\attr \ci \bRd$ be a self-similar set satisfying OSC. Then \attr is nontrivial if and only if \attr has Minkowski dimension %(or equivalently, Hausdorff dimension)
  strictly less than $d$.
\end{prop}
All self-similar sets considered here are assumed to be nontrivial, and the discussion of a self-similar tiling \tiling implicitly assumes that the corresponding attractor \attr is nontrivial and satisfies OSC.

Denote the set of all finite \emph{words} formed by the alphabet $\{1,\dots,N\}$ by
\linenopax
\begin{equation}\label{eqn:def:words}
  \Words := \bigcup_{k=0}^\iy \{1,\dots,N\}^k\,.
\end{equation}
For any word $w=w_1 w_2\dots w_n \in \Words$, let $r_w := r_{w_1}\cdot\ldots\cdot r_{w_n}$ and $\simt_w := \simt_{w_1} \circ \dots \circ \simt_{w_n}$. In particular, if $w \in \Words $ is the \emph{empty word}, then $r_w=1$ and $\simt_w=\mathrm{Id}$.

\begin{defn}\label{def:self-similar-tiling}(Self-similar tiling)
  Let $O$ be a feasible open set for $\simtset$. Denote the connected components of the open set $O \setminus \simt(\cj O)$ by $G_q, q \in Q$, where we assume $Q$ is finite. The sets $G_q$ are called the \emph{generators} of the tiling.
  Then the \emph{self-similar tiling} \tiling associated with the IFS $\{\simt_1,\ldots,\simt_N\}$ and $O$ is the set
  \linenopax
  \begin{equation}\label{eqn:def:self-similar-tiling}
    \tiling(O) := \{ \simt_w(G_q) \suth w \in \Words, q \in Q\}.
  \end{equation}
\end{defn}
We order the words $w^{(1)}, w^{(2)}, \ldots$ of \Words in such a way that the sequence $\sL=\{\scale_j\}_{j = 1}^\iy$ given by $\scale_j := r_{w^{(j)}}$, $j=1,2,\ldots$, is nonincreasing.%
   %In this context, the mapping $\simm_j$ appearing in Definition~\ref{def:fractal-spray} corresponds to $\simt_{w^{(j)}}$.

The terminology ``self-similar tiling'' comes from the fact (proved in \cite[Thm.~5.7]{GeometryOfSST}) that $\tiling(O)$ is an \emph{open tiling} of $O$ in the following sense: The \emph{tiles} $\simt_w(G_g)$ in $\tiling(O)$ are pairwise disjoint open sets and the closure of their union is the closure of $O$, that is,
\linenopax
\[\cj{O}= \cj{\bigcup\nolimits_{q \in Q} \bigcup\nolimits_{w \in \Words} \simt_w(G_q)}\,.\]

 This clarifies that a self-similar tiling (with a single generator) is just a specially constructed fractal spray. (With more than one generator, it is, in fact, a collection of fractal sprays, each with the same fractal string $\sL = \{\scale_j\}_{j=1}^\iy$ and a different generator $\gen_q$, $q \in Q$. It may also be viewed as a fractal spray generated on the union set $\bigcup_{q\in Q} G_q$, as the connectedness of the generator is not a requirement for fractal sprays.)
 
\begin{remark}\label{rem:1-is-enough}
   For self-similar tilings with more than one generator, one can consider each generator independently, and a tube formula of the whole tiling is then given by the sum of the expressions derived for each single generator. Thus, there is no loss of generality in considering only the case of a single generator, which we will denote by $G$ in the sequel. See, however, Remark~\ref{rem:lattice-detail} and Remark~\ref{rem:nonlattice=nonlattice} for further discussion of this issue.
\end{remark}

\begin{defn}
  \label{def:Steiner-like}
  \label{def:monophase}
  For any bounded open set $\gen \ci \bRd$ let $\genir > 0$ be the inradius (the maximal radius of a metric ball contained in the set), and denote the volume of the inner $\ge$-parallel set $\gen_{-\ge}:=\{x\in \gen \suth d(x,\gen^c)\le \ge\}$ by $V(\gen,\ge)$, for any $\ge\geq 0$. 
  
  A \emph{Steiner-like representation} of $ V(\gen,\ge)$ is an expression of the form
  \linenopax
  \begin{align}\label{eqn:def-prelim-Steiner-like-formula}
    V(\gen,\ge) = \sum_{k=0}^{d} \crv_k(\gen,\ge) \ge^{d-k},
    \qq\text{ for } 0 < \ge \leq \genir,
  \end{align}
  where for each $k=0,1,\dots,d$, the coefficient function $\crv_k(\gen,\cdot)$ is a real-valued function on $(0,\genir]$ that is bounded on $[\ge_0, \genir]$ for every fixed $\ge_0 \in (0,\genir]$. Note that Steiner-like representations are not unique.
%\end{defn}
%\begin{defn}\label{def:monophase}
    \gen is said to be \emph{monophase}  if and only if  there is a Steiner-like representation for \gen in which the coefficients $\gk_k(G,\ge)$ are constant, i.e., independent of $\ge$. In this case, we write the coefficients in (2.8) as $\gk_k(G)$ instead of $\gk_k(G,\ge)$. In other words, $\gen$ is monophase, if, in the interval $[0,g]$, $V(\gen,\cdot)$ can be represented as a polynomial of degree at most $d$. Since one always has $\lim_{\ge \to 0^+} V(\gen,\ge) = 0$, it follows that $\crv_d(\gen)=0$ in the monophase case. Moreover, a monophase representation is unique in case it exists. See also Remark~\ref{rem:fractal-boundaries}.
\end{defn}

\begin{figure}%[b]
  \centering
  \scalebox{0.7}{\includegraphics{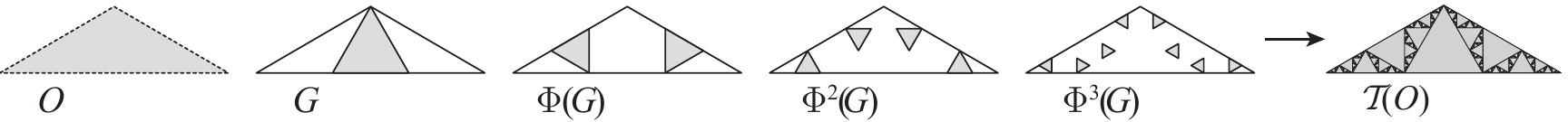}} \\
  \scalebox{0.80}{\includegraphics{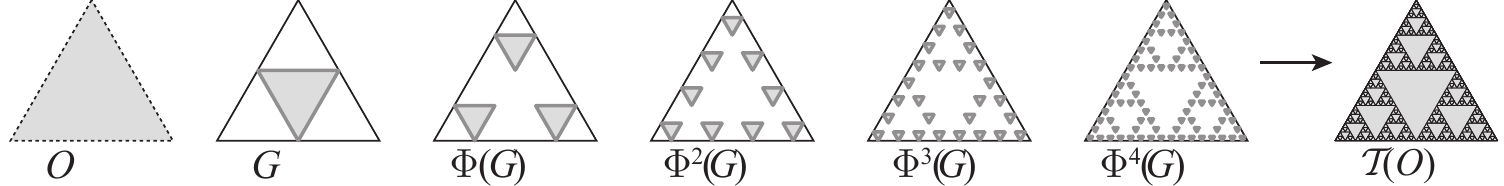}} \vstr[12]\\
  \scalebox{0.82}{\includegraphics{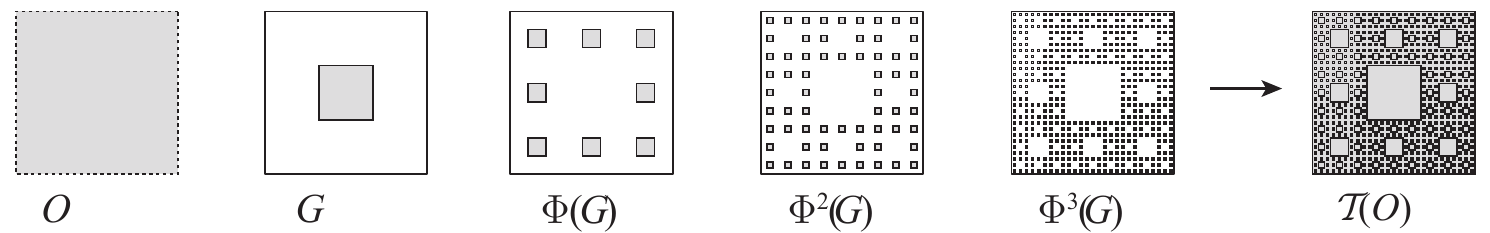}} \vstr[15]
  \caption{\captionsize From top to bottom: the Koch curve tiling, the Sierpinski gasket tiling, and the Sierpinski carpet tiling. In each of these examples, the set $O$ is the interior of the convex hull of \attr, and the generator $G$ is monophase. The Koch curve tiling does not satisfy the compatibility criterion (the hypothesis of Theorem~\ref{thm:compatibility}) but the other two examples do.}
  \label{fig:examples}
\end{figure}

\begin{remark}\label{rem:pluriphase}
Some examples of self-similar tilings associated to familiar fractal sets are shown in Figure~\ref{fig:examples}. In each case, there is a single monophase generator.
At the time of writing, there is no known characterization of monophase generators in terms of the self-similar system $\{\simt_n\}_{n=1}^N$. 
However, it is known from \cite{Kocak2} that a convex polytope in \bRd is monophase (with
Steiner-like function of class $C^{d-1}$) iff it admits an
inscribed $d$-dimensional Euclidean ball (i.e., a $d$-ball tangent to each facet). This includes regular polygons in \bRt and regular polyhedra in \bRd, as well as all triangles and higher-dimensional simplices. Furthermore, it was recently shown in \cite{Kocak2} that (under mild conditions), any convex polyhedron in \bRd ($d \geq 1$) is pluriphase, thereby resolving in the affirmative a conjecture made in \cite{TFSST,TFCD,Pointwise}. Recall from \cite{TFCD,Pointwise} that a set is said to be \emph{pluriphase} iff it admits a Steiner-like representation which is piecewise polynomial, i.e., that $(0,\genir)$ can be partitioned into finitely many intervals with $\crv_k(\gen,\ge)$ constant on each interval.
 We refer to \cite{Kocak2} for further relevant interesting results.
\end{remark}

\section{Zeta functions and fractal tube formulas}
\label{sec:Self-similar-tilings}
%\section{Exact pointwise tube formulas}
\label{sec:Exact-pointwise-tube-formulas}

From now on, let $\tiling = \tiling(O)$ be a self-similar tiling associated with the self-similar system $\{\simt_n\}_{n=1}^N$ and some fixed open set $O$. (We suppress dependence on $O$ when the context precludes confusion.) We refer to the fractal \attr as the self-similar set associated to \tiling.
  Without loss of generality, we continue to assume that there is only a single generator (see Remark~\ref{rem:1-is-enough}). We may also assume that the scaling ratios $\{r_n\}_{n=1}^N$ of $\{\simt_n\}_{n=1}^N$ are indexed in descending order, so that
  \linenopax
  \begin{align}\label{eqn:scaling-ratios-ordered}
    0 < r_N \leq \dots \leq r_2 \leq r_1 < 1.
  \end{align}
Note that there exist $\gs_-,\gs_+ \in \bR$ such that
  \linenopax
  \begin{align}\label{eqn:Moran-bounds}
    0 < \sum_{n=1}^N r_n^\gs < 1,
    \qq \text{for all } \q \gs_- < \gs < \gs_+.
  \end{align}

\begin{defn}\label{def:scaling-zeta-function}
  For a self-similar tiling \tiling, the \emph{scaling zeta function} \gzL is the meromorphic extension to all of \bC of the function defined by
  \linenopax
  \begin{align}\label{eqn:gzL-extended-to-C}
    \gzL(s) = \frac1{1 - \sum_{n=1}^N r_n^s},
    \qq \text{for \q} \gs_- < \Re(s) < \gs_+.
  \end{align}
\end{defn}

The reader familiar with \cite{FGCD}, \cite{TFCD}, or \cite{Pointwise} will notice that \eqref{eqn:gzL-extended-to-C} is the special case of the geometric zeta function of an (ordinary) fractal string when the string is self-similar; see \cite[Thm.~2.9]{FGCD} (or \cite[Thm.~4.7]{TFCD}).

\begin{defn}\label{def:tubular-zeta-fn}
  The \emph{tubular zeta function} of a self-similar tiling with a monophase generator is
  \linenopax
  \begin{align} \label{eqn:volume-zeta-split}
    \gzT(\ge,s)
    &= \gzT[\tail](\ge,s)
    = \frac{\ge^{d-s}\gzL(s)}{d-s} \left(\sum_{k=0}^{d-1} \frac{\genir^{s-k}}{s-k} (d-k)\crv_k(\gen)\right).
  \end{align}
  %Here, \gzL is the \emph{scaling zeta function} given by
  %\begin{equation}\label{eqn:scaling-zeta}
  %  \gzL(s) = \sum_{j=1}^\iy \scale_j^s,
  %  \qq\text{for } s \in \bC \text{ with } \Re(s) > \abscissa,
  %\end{equation}
  %where \abscissa is the abscissa of convergence of this Dirichlet series; see \cite[Def.~3.1]{Pointwise}.%
  %\footnote{As usual, we still denote by $\gzL(s)$ the meromorphic continuation of the series defined by \eqref{eqn:scaling-zeta}. In the present case of self-similar tilings, this continuation exists for all $s \in \bC$ (and is unique); see \cite[Thm.~2.4]{FGCD}.}
\end{defn}

\begin{defn}\label{def:scaling-complex-dimensions}
  The set $\sD_\sL$ of \emph{scaling complex dimensions} of \tiling consists precisely of the poles of \eqref{eqn:gzL-extended-to-C}; that is,
  \linenopax
  \begin{align}\label{eqn:Moran}
    \sD_\sL := \{s \in \bC \suth \textstyle\sum_{n=1}^N r_n^s = 1\}.
  \end{align}
  We define the set $\DT$ of \emph{complex dimensions of the self-similar tiling} \tiling to be
  \linenopax
  \begin{align}\label{eqn:dimns-of-C}
    \DT := \sD_\sL \cup \{0,1,\dots,d\}.
  \end{align}
\end{defn}

The following definition is excerpted from \cite[\S5.3]{FGCD}.

\begin{defn}
  \label{def:screen}
  Let $D<\iy$ denote the abscissa of convergence of \gzL (see Remark~\ref{rem:similarity-dimn}), and choose $f:\bR \to (-\iy, \abscissa]$ to be a bounded Lipschitz continuous function.
  The \emph{screen} is $\Sf = \{f(t) + \ii t \suth t \in \bR\}$, the graph of $f$ with
  the axes interchanged. Here and henceforth, we denote the imaginary unit by $\ii := \sqrt{-1}$. 
  %We let
  %\linenopax
  %\begin{align}
  %  \label{eqn:infS}
  %  \inf S &:= \inf\nolimits_{t \in \bR} S(t) = \inf\{\Re s \suth s \in S\}, \text{ and} \\
  %  \label{eqn:supS}
  %  \sup S &:= \sup\nolimits_{t \in \bR} S(t) = \sup\{\Re s \suth s \in S\}.
  %\end{align}
  The screen is thus a vertical contour in \bC. The region to the
  right of the screen is the set \Wf, called the \emph{window}:
  \linenopax
  \begin{align}
    \label{eqn:window}
    \Wf &:= \{z \in \bC \suth \Re z \geq f(\Im z)\}.
  \end{align}
  For a given string \sL, we always choose $f$ so that \Sf avoids $\sD_{\sL}$ and such that \gzL can be meromorphic{ally} continued to an open neighborhood of \Wf. We also assume (as above) that $\sup f \leq \abscissa$, that is, $f(t) \leq \abscissa$ for every $t \in \bR$. The \emph{visible complex dimensions} are those scaling complex dimensions which lie in the window; this is denoted by
  \linenopax
  \begin{align}\label{eqn:visible-complex-dimensions}
    \DL(\Wf) := \DL \cap \Wf.
  \end{align}
  For the remainder of the paper, we will suppress dependence on $f$ and write simply $S=\Sf$ and $W=\Wf$ for the screen and window.
\end{defn}

In \cite[Thm.~4.1]{Pointwise}, a rather general pointwise tube formula (with and without error term) has been formulated for fractal sprays, which strengthens and extends the distributional tube formulas obtained in \cite{TFCD, TFSST} and generalizes the tube formulas for fractal strings in \cite{FGCD} to higher dimensions.  Various other versions (more specific and more explicit) have been derived from this general tube formula in \cite{Pointwise}, in particular for self-similar tilings. For this note, we only need a formula with error term formulated in \cite[Cor.~5.13]{Pointwise} for self-similar tilings with a single monophase generator; recall the definition of $T$ and $V(T,\ge)$ from \eqref{eqn:tiling-tube}.

\begin{theorem}[Fractal tube formula, with error term, {\cite[Cor.~5.13]{Pointwise}}]
  \label{thm:Fractal-tube-formula-with-error-term}
  Let \tiling be a self-similar tiling as defined above with a single monophase generator $\gen \ci \bRd$, let  $\sL = \{\scale_j\}_{j=1}^\iy$ be the associated fractal string.
  Let $S$ be a screen which avoids the integer dimensions $\{0,1,\dots,d\}$ and for which the visible poles of the tubular zeta function are simple (which implies that $\DL(W)$ and $\{0,1,\dots,d\}$ are disjoint). Then, for all $\ge \in (0,\genir]$, we have the following pointwise formula:
  \linenopax
  \begin{align}\label{eqn:Fractal-tube-formula-with-error-term-simplified}
    V(T,\ge)
    = \sum_{\pole \in \DL(W)} \negsp[10] c_\pole \ge^{d-\pole}
    + \sum_{k \in \{0,1,\dots,d\} \cap W} \negsp[20] c_k \ge^{d-k}
    + \R(\ge)\,,
  \end{align}
  where the coefficients $c_\pole$ and $c_k$ appear in the residues of \gzT at the visible complex and integer dimensions, respectively, and are explicitly given by
  \linenopax
  \begin{align}
    c_\pole &:= \frac{\res[\gw]{\gzL(s)}}{d-\pole}
    \sum_{k=0}^{d-1} \frac{\genir^{\pole-k}(d-k)}{\pole-k} \crv_k(\gen),
      && \text{for }\, \pole \in \sD_\sL(W), \text{ and}
      \label{eqn:self-similar-pointwise-tube-formula-coefficients-cw}\\
    c_k &:= \crv_k(\gen)\gzL(k),
      && \text{for }\, k \in \{0,1,\dots,d-1\}.
      \label{eqn:self-similar-pointwise-tube-formula-coefficients-ck}
    %e_k(\ge) &:= \sum_{j=1}^{J(\ge)} \scale_j^k
    %  \left(\crv_k(\gen, \scale_j^{-1}\ge) - \crv_k(\gen)\right),
    %  && \text{for }\, k \in \{0,1,\dots,d\},
    %  \label{eqn:self-similar-pointwise-tube-formula-coefficients-ek}
  \end{align}
  %with $J(\ge) := \max\{j \geq 1 \suth \scale_j^{-1} \ge < \genir\} \vee 0$.
  %
  Furthermore, the error term in \eqref{eqn:Fractal-tube-formula-with-error-term-simplified} is %given explicitly by
  \linenopax
  \begin{equation} \label{eqn:pointwise-error}
    \R(\ge)
    = \frac{1}{2\gp\ii} \int_S \gzT(\ge,s) \,ds,
    %\frac{\ge^{d-s}\gzL(s)}{d-s} \left(\sum_{k=0}^{d-1} \frac{\genir^{s-k}}{s-k} (d-k)\crv_k(\gen)\right)
      %\ge^{d-s} \gzL(s) \sum_{k=0}^{d-1}\frac{\crv_k(\gen) (d-k)}{(s-k)(d-s)}
  \end{equation}
  and is estimated by $\R(\ge) = O(\ge^{d-\sup f})$ as $\ge \to  0^+$, where $f$ is the function defining $S$.
  %
  %Moreover, if \gen is assumed to be monophase, then \eqref{eqn:Fractal-tube-formula-with-error-term} holds for any screen which avoids the set $\{0,1,\dots,d\}$, and the formula takes the simpler form
  %\linenopax
  %\begin{align}\label{eqn:Fractal-tube-formula-with-error-term-simplified}
  %  V(T,\ge)
  %  = \sum_{\pole \in \DL(W)} \negsp[10] c_\pole \ge^{d-\pole}
  %  + \sum_{k \in \{0,1,\dots,d\} \cap W} \negsp[20] c_k \ge^{d-k}
  %  + \R(\ge)\,,
  %\end{align}
  %with the error term $\R(\ge)$ still as in \eqref{eqn:pointwise-error}.
\end{theorem}

\begin{remark}\label{rem:expork}
  An exposition of some of the main results of \cite{TFCD, Pointwise, Pe2, SST, GeometryOfSST, Minko} (including Theorem~\ref{thm:Fractal-tube-formula-with-error-term}) can be found in \cite{TFSST} and \cite[\S13.1]{FGCD}.
\end{remark}

%\begin{remark}[The simplified hypotheses in Theorem~\ref{thm:Fractal-tube-formula-with-error-term}]
%  \label{rem:hypotheses}
%  Since we are considering self-similar tilings (instead of the full generality of fractal sprays, as in \cite{TFCD,Pointwise}), several standard hypotheses are automatic. This includes the languidity condition (self-similar strings are always strongly languid; see \cite[proof of Thm.~5.7]{Pointwise}) and choice of screen described above (so that the window $W$ contains the integers $\{0,1,\ldots,d\}$). Note also that $\abscissa < d$ follows automatically from \cite[Cor.~2.13]{GeometryOfSST}, for self-similar tilings.
%\end{remark}

\section{Minkowski measurability results for self-similar tilings} % and fractals with monophase generators}
%\section{Corollaries concerning Minkowski measurability of self-similar tilings}% following from tube
\label{sec:Minkowski-measurability}

Now we are going to discuss the question of Minkowski measurability for self-similar tilings. We start by clarifying the notion of Minkowski content for such tilings.  As in \eqref{eqn:tiling-tube}, we write $V(T,\ge)$ for the inner parallel volume of the open set $T$, where $T:=\bigcup_{R\in\sT} R$ is the union set of the tiles of $\sT$ (which are open sets by definition). 
%By the Minkowski content of the tiling \sT we mean, in fact, the \emph{relative Minkowski content of $\bd T$ relative to the set $T$} (in the sense of \cite{Zubrinic,Zubrinic2}). In much of the literature, this is referred to as the (inner) Minkowski content of $\bd \tiling$; see, e.g., \cite{FGCD} and the references contained therein.%in \cite{BroCa, Lap:FD, Lap:Dundee, LaPo1, FleVa, LaPo2, FGNT1, FGCD, TFCD, Pointwise, GeometryOfSST, Zubrinic, Zubrinic2, PokW, Minko, RatajWinter, RatajWinter2}%
%\footnote{See also Definition~\ref{def:outer-Mink-content} and the ensuing discussion.}
%Since the inner $\ge$-parallel set of $T$ is exactly the $\ge$-parallel set of $\bd T$ relative to $T$, that is, $T_{-\ge}=(\bd T)_\ge\cap T$, this can also be formulated as follows.

%We provide definitions of Minkowski measurability and content of self-similar tilings, as some terminology is slightly nonstandard. In particular, the following definitions would typically be taken as the definition of Minkowski measurability of the \emph{boundary} of \tiling. %In \S\ref{sec:fractals-themselves}, we transfer these results for the tiling to the corresponding fractal itself, under a suitable additional hypothesis.

\begin{defn}[Minkowski content and dimension]
  \label{def:Minkowski-content-and-dimension}
  Let \tiling be a self-similar tiling (or a fractal spray) in $\bR^d$ and let $0\le\alpha\le d$. If the limit
  \linenopax
  \begin{align}\label{eqn:Minkowski-content}
    \Mink_\ga(\tiling)
    := \lim_{\ge \to 0^+} \ge^{-(d-\ga)} V(T, \ge),
  \end{align}
  exists (as a value in $[0,\iy]$), then this number is called the \emph{\ga-dimensional Minkowski content} of \tiling.
  Similarly as for sets, $\tiling$ is said to be \emph{Minkowski measurable (of dimension $\alpha$)}, if $\sM_\alpha(\tiling)$ exists and satisfies $0<\sM_\ga(\tiling)<\infty$.
  Furthermore, the \emph{Minkowski dimension} of \tiling is the real number $\dim_M \negsp[4]\tiling \in[0,d]$ given by
  \linenopax
  \begin{align}\label{eqn:Minkowski-dimension}
    \dim_M \negsp[4]\tiling
    :=& \inf\{\ga \geq 0 \suth \Mink_\ga(\tiling) = 0\}.
    %:=& \inf\{\ga \geq 0 \suth V(T,\ge) = O(\ge^{d-\ga}), \text{ as } \ge \to 0^+\}.
  \end{align}
%  Then the \emph{upper (inner) Minkowski content} and the \emph{lower (inner) Minkowski content} of \tiling are given by
%  \linenopax
%  \begin{align}
%    \Minku(\tiling)
%    &:= \limsup_{\ge \to 0^+} V(T,\ge) \ge^{-(d-\abscissa_\tiling)}
%      \text{, and }\label{eqn:upper-Minkowski-content} \\
%  \end{align}
%  is
%  \linenopax
%  \begin{align}
%\label{eqn:lower-Minkowski-content}
%    \Minkl(\tiling)
%    &:= \liminf_{\ge \to 0^+} V(T,\ge) \ge^{-(d-\abscissa_\tiling)},
%  \end{align}
%  respectively. Clearly, one has
%  %\linenopax
%  %\begin{align}\label{eqn:lower<=upper-content}
%  $  0 \leq \Minkl(\tiling) \leq \Minku(\tiling) \leq \iy.$
%  %\end{align}
%  When this takes the form
%  \linenopax
%  \begin{align}\label{eqn:lower=upper-content}
%    0 < \Minkl(\tiling) = \Minku(\tiling) < \iy,
%  \end{align}
%  we say that \tiling is \emph{Minkowski measurable}, and that \tiling has \emph{Minkowski content} $\Mink(\tiling) := \Minkl(\tiling) = \Minku(\tiling)$.
%  More generally, if there is an $\ga \geq 0$ such that the limit
  \linenopax
  %one says that the \emph{$\ga$-dimensional Minkowski content} of \tiling exists.
  %= \inf\{\ga \geq 0 \suth \Mink_\ga(\tiling) = 0\}$. %, then $\Mink_\ga(\tiling)$ is called.

%  In general, discussion of $\Mink_\ga(\tiling)$ tacitly implies that the corresponding limit %\eqref{eqn:Minkowski-content}
 % exists (possibly in $[0,\iy]$).
  %
  %
  %If $0 < \Mink_\ga(\tiling) < \iy$ for some $\ga \geq 0$, then we must have $\ga = \abscissa$ and then \tiling is said to be \emph{Minkowski measurable}.
  %
\end{defn}

%\begin{remark}\label{rem:D=alpha}
%  If the limit in \eqref{eqn:Minkowski-content} exists and satisfies $0 < \Mink_\ga(\tiling) < \iy$, then we are in the situation of \eqref{eqn:lower=upper-content},
%  %. In this case, the Minkowski dimension coincides with the abscissa of convergence of \gzL, i.e.,
%  and one has $\abscissa_\tiling = \ga$.
%\end{remark}

It is obvious that Minkowski measurability of dimension $\alpha$ implies that $\dim_M \negsp[4]\tiling=\alpha$. In analogy with the average Minkowski content for sets or fractal strings, see e.g.~\cite[Def.~8.29]{FGCD}, the next definition will be useful in the case of lattice self-similar tilings, when the Minkowski content does not exist.

\begin{defn}[Average Minkowski content]
  \label{def:Average-Minkowski-content}
  Let \tiling be a self-similar tiling (or fractal spray) in $\bR^d$ and  $0\leq\ga \leq d$. If the limit
  \linenopax
  \begin{align}\label{eqn:alpha-average-Minkowski-content}
    \Minka_\ga (\tiling)
    := \lim_{b \to \iy} \frac1{\log b} \int_{1/b}^1 \ge^{-(d-\ga)} V(T,\ge) \frac{d\ge}{\ge}
  \end{align}
  exists in $[0,\iy]$, % and takes a value in $(0,\iy)$,
  then $\Minka_\ga (\tiling)$ is called the \emph{\ga-dimensional average Minkowski content} of \tiling. % and is said to exist.
  The generic term \emph{average Minkowski content} refers to the (only interesting) case when $0 < \Minka_{\ga}(\tiling) < \iy$ for $\ga=\dim_M \negsp[4]\tiling$.
\end{defn}

%\subsection{Monophase generators}

%If there are multiple generators, then they are disjoint by construction. As usual, we can thus work with each one separately.

\begin{remark}[Various incarnations of \abscissa]
  \label{rem:similarity-dimn}
  We use the symbol \abscissa to denote the \emph{abscissa of convergence of \gzL}:
  \linenopax
  \begin{align}\label{eqn:}
    \abscissa := \inf\{\Re(s) \suth |\gzL(s)|<\infty\}.
  \end{align}
  This abscissa is analogous to the radius of convergence of a power series; the Dirichlet series $\sum_{j=1}^\iy \ell_j^s$ converges if and only if $\Re(s) > \abscissa$, in which case it converges absolutely. 
  It is clear from \eqref{eqn:Moran-bounds} that \abscissa exists and is both positive and finite.

 It follows from \cite[Thm.~3.6]{FGCD} that $\abscissa$ is a simple pole of \gzL and that $D$ is the only pole of \gzL (i.e., the only scaling complex dimension of \tiling) which lies on the positive real axis. Furthermore, it coincides with the unique real solution of \eqref{eqn:Moran}, often called the \emph{similarity dimension} of \attr and denoted by $\gd(\attr)$. Since \attr satisfies OSC, $D$ also coincides with the Minkowski and Hausdorff dimension of \attr, denoted by $\dim_M \negsp[4]\attr$ and $\dim_H \negsp[4]\attr$, respectively. (For this last statement, see \cite{Hut}, as described in \cite[Thm.~9.3]{Falconer}.)
  Moreover, it is clear that $\abscissa>0$ since $N \geq 2$, and that $\abscissa \leq d$; in fact, Proposition~\ref{cor:OSC-dimension-d-implies-trivial} implies $\abscissa < d$. In summary, we have
  \linenopax
  \begin{align}\label{eqn:D-dimensions}
    0 < \abscissa < d
    \qq\text{and}\qq
    \abscissa = \gd(\attr) = \dim_M \negsp[4]\attr = \dim_H \negsp[4]\attr.
  \end{align}
  In Theorem~\ref{thm:Mink-meas} below we establish that also $\dim_M \negsp[4]\tiling$ coincides with these numbers.
\end{remark}

The following result is an immediate consequence of \cite[Thm.~3.6]{FGCD}, which provides the structure of the complex dimensions of self-similar fractal strings (even for the case when \abscissa may be larger than 1).

\begin{prop}[Lattice/nonlattice dichotomy, see {\cite[\S4.3]{TFCD}}]
  \label{thm:lattice-nonlattice-dichotomy}
  \hfill \\
	  \emph{(Lattice case)}. When there is an $r>0$ such that each scaling ratio $r_n$ can be written as $r_n = r^{k_n}$ for some integer $k_n$, then the scaling complex dimensions lie periodically on finitely many vertical lines, including the line $\Re s = \abscissa$. This means that $\{\Re(s) \suth s \in \DL\}$ is a finite set, and there is a number $\per>0$ (called the \emph{oscillatory period}) such that for any integer $m \in \bZ$, $s+\ii m \per \in \DL$ whenever $s \in \DL$. Consequently, there are clearly infinitely many complex dimensions with real part \abscissa.
  \hfill \\
  \emph{(Nonlattice case)}. Otherwise, the scaling complex dimensions are quasiperiodically distributed (as described in \cite[\S3]{FGCD}) and $s = \abscissa$ is the only complex dimension with real part \abscissa. However, there exists an infinite sequence of simple scaling complex dimensions approaching the line $\Re s = \abscissa$ from the left.
  In the generic nonlattice case (that is, when the distinct scaling ratios generate a group of maximal rank), the set $\{\Re s \suth s \in \DL\}$ appears to be dense in finitely many compact subintervals of $[\gs_-,\gs_+]$, where $\gs_-, \gs_-$ are as in \eqref{eqn:Moran-bounds}; cf.~\cite[\S3.7.1]{FGCD}.
\end{prop}

The proof of Theorem~\ref{thm:Mink-meas} below is based on this lattice/nonlattice dichotomy; see also \cite[Prop.~5.5]{Pointwise} or \cite[\S3]{FGCD}.

\begin{remark}\label{rem:dichotomy}
  It follows from \cite[Thm.~3.6]{FGCD} that in the lattice case, % (i.e., when $r_n = r^{k_n}$, $n=1,\ldots,N$, for some $0<r<1$ and positive integers $\{k_n\}_{n=1}^N$),
  each scaling complex dimension (i.e., each pole \gw of \gzL) has the same multiplicity (and Laurent expansion with the same principal part) on each vertical line, and that each has real part satisfying $\Re \gw \leq \abscissa$.
  In particular, since $\abscissa$ is simple (see Remark~\ref{rem:similarity-dimn}), all the scaling complex dimensions $\{\abscissa+\ii m\per\}_{m \in \bZ}$ (where $\per = 2\gp/\log r^{-1}$) along the vertical line $\Re s = \abscissa$ are simple and have residue given by
  \linenopax
  \begin{align}\label{eqn:lattice-residues}
    \res[\abscissa]{\gzL(s)}
    = \frac{1}{\log r^{-1} \sum_{n=1}^N k_n r^{k_n \abscissa}}.
  \end{align}
  In the nonlattice case, $\abscissa$ is simple with residue
  \linenopax
  \begin{align}\label{eqn:nonlattice-residues}
    \res[\abscissa]{\gzL(s)}
    = \frac{1}{\sum_{n=1}^N r_n^{\abscissa} \log r_n^{-1}}.
  \end{align}
  Note that \eqref{eqn:nonlattice-residues} is also valid in the lattice case. %Proposition~\ref{thm:lattice-nonlattice-dichotomy} and the contents of this remark are used when applying Theorem~\ref{thm:ptwise-result-self-similar-case} and Corollary~\ref{thm:ptwise-result-self-similar-case-simplified} to the examples in Section~\ref{sec:Examples}.
\end{remark}

\begin{defn}\label{def:spam}
  Given $s \in \bC$, let
  \linenopax
  \begin{align}\label{eqn:Md(G)}
    \spam_s(\gen)
    := \sum_{k=0}^{d-1} \frac{\genir^{s-k}}{s-k} (d-k) \crv_k(\gen).
  \end{align}
\end{defn}

The sum extends only to $d-1$ in \eqref{eqn:Md(G)} because one has $\crv_d \equiv 0$ in the monophase case, as noted in Definition~\ref{def:monophase}. %, and $\gzT[\head](\ge,s) \equiv 0$ as noted in Remark~\ref{rem:monophase-head-vanish}.

\begin{remark}\label{rem:Gamma}
Using that $G$ is monophase, the Steiner-like representation \eqref{eqn:def-prelim-Steiner-like-formula} (with $\gk_d=0$), and the relation $V(G,g)=\sum_{k=0}^{d-1} g^{d-k} \gk_k(G)$ (from \eqref{eqn:def-prelim-Steiner-like-formula}) %\cite[(2.7)]{Pointwise}), 
it is not difficult to see that for any real $\ga \in (d-1,d)$, one has
\linenopax
\begin{align*}
  \int_0^\infty \ge^{\ga-d} V(G,\ge) \frac{d\ge}\ge 
  &= \int_0^g \sum_{k=0}^{d-1} \ge^{\ga-k-1} \gk_k(G) d\ge +\int_g^\infty \ge^{\ga-d-1} V(G,g) d\ge \\  
  &=  \sum_{k=0}^{d-1} \frac{g^{\ga-k}}{\ga-k} \gk_k(G) - \sum_{k=0}^{d-1} \crv_k(G) \frac{g^{\ga-k}}{\ga-d}  \\
  %&= \sum_{k=0}^{d-1} \frac{g^{\ga-k}}{\ga-k} \gk_k(G) - \lambda_d(G) \frac{g^{\ga-d}}{\ga-d} \\
  &= \sum_{k=0}^{d-1} g^{\ga-k} \gk_k(G)\left(\frac 1{\ga-k}+\frac 1{d-\ga}\right) \\
  &= \frac 1{d-\ga} \Gamma_\ga(G).
\end{align*}
Since the volume $V(G,\cdot)$ is clearly a strictly positive function on $(0,\infty)$, this computation shows that $\Gamma_\ga(G)>0$ for any real $\ga \in (d-1,d)$. In particular, $\Gamma_D(G)>0$ under the hypothesis of Theorem~\ref{thm:Mink-meas} below. See also \cite[Rem.~4]{Kocak1}.
\end{remark}

\begin{theorem}[Minkowski measurability of self-similar tilings, monophase case]
  \label{thm:Mink-meas}
  Suppose a self-similar tiling \tiling in \bRd has a single monophase generator $G$ and the abscissa of convergence $D$ of the associated scaling zeta function $\gz_\sL$  satisfies $d-1<D$.
  %Assume $\spam_\abscissa(\gen) \neq 0$ and (in the lattice case) that $\spam_{\abscissa+\ii m \per}(\gen) \neq 0$ for some $m \in \bZ \less\{0\}$, where \per is the oscillatory period as in Remark~\ref{rem:dichotomy}.
  Then
  %$\spam_\abscissa(\gen) > 0$ and
  %the Minkowski dimension of \tiling coincides with the abscissa of convergence of \gzL, i.e., 
  $\dim_M \negsp[4]\tiling = \abscissa$. Moreover, \tiling is Minkowski measurable if and only if $\gz_\sL$ is nonlattice.
  %Suppose \tiling has a monophase generator and $d-1 < \abscissa$, and assume that $\spam_\abscissa(\gen) \neq 0$. Then $F$ is Minkowski measurable if and only if the associated scaling zeta function \gzL is nonlattice.
  In this case, the Minkowski content of \tiling is given by %the finite and positive number
  \linenopax
  \begin{align}\label{eqn:nonlattice-content}
    \Mink_{\abscissa}(\tiling) = \frac{\spam_\abscissa(\gen)}{(d-\abscissa)\sum_{n=1}^N r_n^{\abscissa} \log r_n^{-1}},
  \end{align}
  where $\spam_D(\gen)$ is as in \eqref{eqn:Md(G)}. % and $\res[\abscissa]{\gzL(s)}$ is given by the first equality in \eqref{eqn:lattice-residues} below. 
  Moreover, $0 < \Mink_{\abscissa}(\tiling) < \iy$. 
  In the lattice case, the Minkowski content of \tiling does not exist, but the \emph{average} Minkowski content $\Minka_\abscissa(\tiling)$ exists and 
  \linenopax
  \begin{align}\label{eqn:lattice-content}
    \Minka_\abscissa(\tiling) = \frac{\spam_\abscissa(\gen)}{(d-\abscissa) \sum_{n=1}^N r^{k_n \abscissa} k_n  \log r^{-1} }.
  \end{align}
  Furthermore, $0 < \Minka_{\abscissa}(\tiling) < \iy$. 
  %agrees with \eqref{eqn:nonlattice-content}, with the factor $\res[\abscissa]{\gzL(s)}$ given by either equality in \eqref{eqn:lattice-residues}.
\end{theorem}
\begin{proof}%[Proof of Theorem~\ref{thm:Mink-meas}]
First, note that $d-1 < \abscissa < d$; the first inequality holds by hypothesis, and the second holds for any self-similar tiling, by \cite[Cor.~2.13]{GeometryOfSST}.

\emph{Lattice case.}
In this case, the scaling ratios of the similarity mappings are $r_n = r^{k_n}$, $n=1,\ldots,N$, for some $0<r<1$ and positive integers $\{k_n\}_{n=1}^N$. Moreover, the complex dimensions of \sL are periodically distributed with period $\per ={2\gp}/{\log r^{-1}}$ along finitely many vertical lines, the rightmost of which is the line $\Re s = \abscissa$; see Remark~\ref{rem:dichotomy} above and \cite[Thm.~3.6]{FGCD}.
%in the sense that there is a unique smallest number $\per>0$ such that for any $\gw \in \bC$ with $\Im \gw \geq 0$,
%\[\gw \in \sD_\sL \q\implies\q \gw+\ii m \per \in \sD_\sL,
%  \qq\text{for all } m=1,2,\dots.\]
  %Since all the complex dimensions lie on finitely many vertical lines in the lattice case,
  One can therefore take the screen $S$ to be any vertical line in \bC of the form $\Re s = \gq$, for which (i) $d-1 < \gq < \abscissa$, and (ii) the only scaling complex dimensions of \tiling in the window $W$ are those with $\Re s = \abscissa$.
  We then apply the methods of proof of \cite[Thm.~8.23]{FGCD} and \cite[Thm.~8.30]{FGCD}.   The tubular zeta function in \cite{TFCD,Pointwise} is different from the tubular zeta function corresponding to the 1-dimensional case considered in \cite{FGCD}, but they have similar forms. (Some discussion of this issue is provided in \cite[\S13.1]{FGCD}.)

For clarity, we will now explain in detail how to complete the proof in the present special case of a self-similar tiling in \bRd with a monophase generator.
  %More specifically, choose a number $\gQ \in (d-1,\abscissa)$ such that the first line of poles of \gzL strictly to the left of \abscissa lies strictly to the right of
In view of \cite[Rem.~5.6]{Pointwise}, all the poles of \gzL on the line $\Re s = \abscissa $ are simple, including \abscissa itself. Moreover, by \cite[Rem.~5.6]{Pointwise} or \cite[Thm.~2.16]{FGCD} (see Remark~\ref{rem:dichotomy}), for each $m\in \bZ$, the residue of \gzL at the pole $\abscissa + \ii m \per$ is independent of $m$ and equal to
  \linenopax
  \begin{align}\label{eqn:lattice-residues}
    \res[\abscissa]{\gzL(s)}
    = \frac{1}{\sum_{n=1}^N r_n^{\abscissa} \log r_n^{-1}}
    = \frac{1}{\log r^{-1} \sum_{n=1}^N k_n r^{k_n \abscissa}}.
  \end{align}
  More specifically, we apply Theorem~\ref{thm:Fractal-tube-formula-with-error-term} with the aforementioned choice of screen $S = \{\Re s = \gq\}$, where $d-1 < \gq < \abscissa$, to obtain
  \linenopax
  \begin{align}\label{eqn:V-in-terms-of-xG}
    V(T,\ge) = \ge^{d-\abscissa} \xG\left(\log_{r^{-1}}(\ge^{-1})\right)
    +O(\ge^{d-\gq}), \q\text{as } \ge \to 0^+,
  \end{align}
  where \xG is the \bR-valued periodic function (of period 1) on \bR given by the following absolutely convergent Fourier series expansion:
  \linenopax
  \begin{align}
    \xG(x)
    &= \res[\abscissa]{\gzL(s)} \sum_{m \in \bZ} \sum_{k=0}^{d-1}
      \frac{\genir^{\abscissa + \ii m \per - k}(d-k)}
      {(\abscissa + \ii m \per - k)(d-\abscissa - \ii m \per)}
      \crv_k(\gen) e^{2\gp \ii m x} \notag \\
    &= \res[\abscissa]{\gzL(s)} \sum_{m \in \bZ}
      \frac{\spam_{\abscissa+\ii m \per}(\gen)}{d-\abscissa - \ii m \per}
      e^{2\gp \ii m x}, %\notag \\
    %&= \res[\abscissa]{\gzL(s)} \sum_{m \in \bZ}
    %  \frac{\spam_{\abscissa+\ii m \per}(\gen)}{d-\abscissa - \ii m \per}
    %  e^{2\gp \ii m x}
    \label{eqn:xG}
  \end{align}
  where we used \eqref{eqn:Md(G)} in the last equality. Note that the  periodic function \xG is nonconstant if and only if there is some $m \in \bZ \less\{0\}$ for which the \nth[m] Fourier coefficient is nonzero.   Observe that $\spam_{\abscissa+\ii m \per}(\gen) \neq 0$ for some $m \in \bZ \less\{0\}$ if and only if $\spam_{\abscissa+\ii m \per}(\gen) \neq 0$ for some $m \geq 1$, since the periodic function \xG is \bR-valued.
  In light of \eqref{eqn:xG}, and since $\res[\abscissa]{\gzL(s)} \neq 0$ by \eqref{eqn:lattice-residues}, this occurs if and only if
  \linenopax
  \begin{align*}%\label{eqn:}
    \spam_{\abscissa+\ii m \per}(\gen)
    = \sum_{k=0}^{d-1} \frac{\genir^{{\abscissa+\ii m \per}-k}}{{\abscissa+\ii m \per}-k} (d-k) \crv_k(\gen)\neq 0
    \qq \text{ for some } m \in \bZ \less\{0\}.
  \end{align*}
  The validity of this last condition is seen as follows:
  first observe that
  \linenopax
  \begin{align*}%\label{eqn:}
  \Gamma_s(G)
  := g^s\sum_{k=0}^{d-1}\frac{\genir^{-k}}{{s-k}} (d-k) \crv_k(\gen)
   = g^s \frac{p(s)}{s(s-1) \dots (s-d+1)},
  \end{align*}
  where $p$ is some polynomial of degree at most $d-1$. See also \cite[Rem.~3]{Kocak1} for a closely related computation. Since $\Gamma_s(G)=0$ only if $p(s)=0$, we conclude that $\Gamma_s(G)$ has at most $d-1$ zeros. This implies in particular that  $\spam_{\abscissa+\ii m \per}(\gen)$ cannot be zero for all $m\in\bZ\setminus\{0\}$.
  Recall also that the \nth[0] Fourier coefficient of $\xG$ is positive because $\spam_\abscissa(\gen) > 0$ by Remark~\ref{rem:Gamma}.

  Since $\theta <\abscissa$, we have
  \linenopax
  \begin{align}\label{eqn:O-o}
    O(\ge^{\abscissa-\gq})=o(1),
    \qq\text{ as } \ge \to 0^+.
  \end{align}
  In combination with \eqref{eqn:V-in-terms-of-xG}, this yields
  \linenopax
  \begin{align}\label{eqn:V-decay}
    \ge^{-(d-\abscissa)} V(T,\ge) = \xG(\log_{r^{-1}}(\ge^{-1})) + o(1),
    \qq\text{ as } \ge \to 0^+.
  \end{align}
  Since \xG has a nonzero \nth[0] Fourier coefficient and is nonconstant, bounded and periodic, it follows from \eqref{eqn:V-decay} and \eqref{eqn:Minkowski-content} that \tiling is not Minkowski measurable. Moreover, \eqref{eqn:V-decay} implies  that \tiling has Minkowski dimension \abscissa.
  Indeed, for any given $\alpha>\abscissa$, the boundedness of $\xG$ implies that $\ge^{-(d-\alpha)} V(T,\ge)$ vanishes as $\ge\to 0^+$. In light of \eqref{eqn:Minkowski-dimension}, this yields $\dim_M \negsp[4]\sT\le \abscissa$. For the reverse inequality, let $\alpha<D$. Since \xG is periodic and not identically zero, we can find a sequence of positive numbers $(\ge_n)$ tending to 0 and some constant $c\neq 0$  such that $\xG(\log_{r^{-1}}(\ge_n^{-1}))=c$ for all $n\in\bN$. Obviously, we have $|\ge_n^{\alpha-\abscissa}\xG(\log_{r^{-1}}(\ge_n^{-1}))| \to \infty$ as $n\to\infty$, which implies the same for the sequence  $\ge_n^{-(d-\alpha)} V(T,\ge_n)$.
   Hence, for any given $\ga < \abscissa$, $V(T,\ge)$ is not $O(\ge^{d-\alpha})$ as $\ge \to 0^+$, which (by \eqref{eqn:Minkowski-dimension}) proves $\dim_M \negsp[4]\sT \geq \abscissa$. 
   The inequality $\dim_M \negsp[4]\sT \geq \abscissa$ can also be derived from the fact that the associated self-similar set $F$ is always a subset of the boundary of $\sT$, which immediately implies $\dim_M \negsp[4]\sT \geq \dim_M \negsp[4]F$ (see \cite[Prop.~6.1 and Rem. 5.12]{GeometryOfSST}). Note also that the Minkowski dimension $\dim_M \negsp[4]\attr$ of \attr coincides with $D$, cf.\ Remark~\ref{rem:similarity-dimn}.
   
   Finally, following the proof of \cite[Thm.~8.30]{FGCD}, we divide \eqref{eqn:V-in-terms-of-xG} by \ge and integrate from $\frac1b$ to 1 with respect to \ge, as in \eqref{eqn:alpha-average-Minkowski-content}. With the change of variables $x=\log_{r^{-1}}(\ge^{-1})$, and in view of \eqref{eqn:O-o}, %
  %\footnote{More specifically, prior to integrating, one should replace $o(1)$ by $O(\ge^{D-\theta})$ on the right side of \eqref{eqn:V-decay}, and note that $D-\gq > 0$. Hence, upon integrating, we obtain \eqref{eqn:G-integral-limit}, with an error term that is $o(1)$, as $b \to \iy$.}
  this yields
  \linenopax
  \begin{align}\label{eqn:G-integral-limit}
    \frac{1}{\log b} \int_{1/b}^1 \ge^{-(d-\abscissa)} V(T,\ge) \frac{d\ge}{\ge} = \frac{1}{\log_{r^{-1}} b} \int_{0}^{\log_{r^{-1}} b} \xG(x) \, dx + o(1),
    \qq \text{as } b \to \iy.
  \end{align}
  Since \xG is 1-periodic, the latter expression tends to $\int_0^1 \xG(x)\,dx$ as $b \to \iy$. In view of Definition~\ref{def:Average-Minkowski-content}, we deduce that the average Minkowski content $\Minka_\abscissa(\tiling)$ exists and is given by the \nth[0] Fourier coefficient of \xG. In other words, \eqref{eqn:xG} yields
  \linenopax
  \begin{align}\label{eqn:avg-Mink(tiling)}
    \Minka_\abscissa(\tiling)
    &= \int_0^1 \xG(x) \, dx
    %= \frac{\res[\abscissa]{\gzL(s)}}{d-\abscissa} \sum_{k=0}^{d-1} \frac{\genir^{\abscissa-k}}{\abscissa-k} (d-k) \crv_k(\gen)
     = \res[\abscissa]{\gzL(s)} \frac{\spam_{\abscissa}(\gen)}{d-\abscissa},
     % \frac{}{d-\abscissa}
  \end{align}
  which coincides with \eqref{eqn:lattice-content}, as claimed. Note that $\spam_\abscissa(\gen) > 0$ by Remark~\ref{rem:Gamma}. %assumption and $\res[\abscissa]{\gzL(s)} > 0$ by \eqref{eqn:lattice-residues}, it follows that $0 < \Minka_\abscissa(\tiling) < \iy$ (as a nonzero limit of positive quantities), and therefore $\spam_\abscissa(\gen) > 0$.
  This concludes the proof of Theorem~\ref{thm:Mink-meas} in the lattice case.

\emph{Nonlattice case.}
In this case, \abscissa is the only pole of \gzL on the line $\Re s = \abscissa$ and we follow the method of proof of \cite[Thm.~8.36]{FGCD}. Because \cite[Thm.~3.25]{FGCD} applies to generalized self-similar strings, it holds even for $\abscissa > 1$. Consequently, the statement and method of proof are applicable in the present context of self-similar tilings in \bRd. %Just as in the previous case, $\crv_d\equiv0$ and $d-1 < \abscissa < d$ and $\crv_d\equiv0$, but
As is recalled in Remark~\ref{rem:similarity-dimn}, \abscissa is simple; see Thm.~2.16 or Thm.~3.6 of \cite{FGCD}. According to \cite[Thm.~3.25]{FGCD}, there exists a screen $S$ which lies to the left of the line $\Re s = \abscissa$ such that \gzL is bounded on $S$, and all the visible scaling complex dimensions are simple and have uniformly bounded residues, in the sense that there is a constant $C>0$ for which
\linenopax
\begin{align}\label{eqn:residue-bound}
  \left|\res[\gw]{\gzL(s)}\right| \leq C,
  \qq \text{for all } \pole \in \DL(W),
\end{align}
where $W$ is the window corresponding to $S$.
In fact, the screen can be taken to be arbitrarily close to (but bounded away from) the line $\Re s = \abscissa$. More precisely, one can choose this screen $S = S_f$ (so $f$ denotes the function defining $S$) such that
\linenopax
  \begin{align}\label{eqn:}
    d-1 < \abscissa - 2\gd < \inf f \leq \sup f < \abscissa - \gd < \abscissa < d,
  \end{align}
for some fixed but arbitrarily small $\gd > 0$.
Except for $\pole = \abscissa$, this ensures all the visible complex dimensions \pole of \tiling lie in
\linenopax
\begin{align}\label{eqn:polestrip}
  \polestrip_\gd := \{s \in \bC \suth \abscissa - 2\gd < \Re s < \abscissa\}.
\end{align}
Since each $\pole \in \DL(W)$ is simple and lies in $P_\gd$, this choice of \gd and $S$ implies that $\DL(W) \cap \{0,1,\dots,d\} = \es$; i.e., that there are \emph{no} visible integer dimensions that are poles of \gzT.
%$\sup f \leq \abscissa - \gd$ and
%\linenopax
%\begin{align}\label{eqn:polestrip-cont}
%  \polestrip_\gd \ci \{s \in \bC \suth d-1 < \Re s < d\}.
%\end{align}
%Since each $\pole \in \DL(W)$ is simple and lies in $P_\gd$ (see \eqref{eqn:polestrip}), we have
% (see \eqref{eqn:polestrip})
%In other words, this choice of \gd and $S$ ensures% in this case. %: $\{0,1,\dots,d\} \cap W = \es$.
Hence, upon application of Theorem~\ref{thm:Fractal-tube-formula-with-error-term} to the above screen $S$, \eqref{eqn:Fractal-tube-formula-with-error-term-simplified} becomes
\linenopax
\begin{align}\label{eqn:Minkowski-characterization-deriv-2}
  V(T,\ge)
  &= c_\abscissa \ge^{d-\abscissa} + \sum_{\pole \in \DL(W) \cap \polestrip_\gd} c_\pole \ge^{d-\pole} +\R(\ge),
\end{align}
where, in light of \eqref{eqn:self-similar-pointwise-tube-formula-coefficients-cw} and \eqref{eqn:Md(G)},
\linenopax
\begin{align}
  c_\pole
  = \res[\gw]{\gzL(s)} \frac{\spam_\pole(\gen)}{d-\pole},
  \qq\text{for } \pole \in \DL(W),
    \label{eqn:nonlattice-Minkowski-content-gen}
\end{align}
including the case when $\gw = \abscissa$.
%since $\DL(W)$ and $\{0,1,\dots,d\}$ are disjoint, and letting $\pole = \abscissa$ gives
%\linenopax
%\begin{align}
%  c_\abscissa
%  &= \res[s=\abscissa]{\gzL(s)} \frac{\spam_\abscissa(\gen)}{d-\abscissa}
%    \label{eqn:nonlattice-Minkowski-content}.
%\end{align}
%in light of
%  \linenopax
%  \begin{align}\label{eqn:gzG[tail]}
%    \gzG[\tail](\ge,s) = \frac{\ge^{d-s}}{d-s} \spam_s(\gen),
%  \end{align}
%  and
%      \linenopax
%      \begin{align}\label{eqn:residues-of-gzT(i)}
%        \gzG[\tail](\ge,\pole) \res{\gzL(s)}
%        = \frac{\ge^{d-\pole}}{d-\pole} \res{\gzL(s)} \spam_\pole(\gen),
%      \end{align}
Note that $c_\abscissa > 0$, since $\spam_\abscissa(\gen) > 0$, by Remark~\ref{rem:Gamma}, and since \abscissa is a pole of \gzL. Indeed, by \cite[Rem.~5.6]{Pointwise} or \cite[Thm.~2.16]{FGCD} (see Remark~\ref{rem:dichotomy}), the residue $\res[\abscissa]{\gzL(s)}$ is given by the first equality of \eqref{eqn:lattice-residues}, and thus $\res[\abscissa]{\gzL(s)} > 0$. In combination with \eqref{eqn:Minkowski-characterization-deriv-2} and the error estimate $\R(\ge) = O(\ge^{d - \sup f})$ as $\ge \to 0^+$ (from Theorem~\ref{thm:Fractal-tube-formula-with-error-term}), this implies
\linenopax
\begin{align}\label{eqn:Minkowski-measurability-claim}
  \ge^{-(d-\abscissa)} V(T,\ge)
  &= c_\abscissa + \sum_{\pole \in \DL(W) \cap \polestrip_\gd} c_\pole \ge^{\abscissa-\pole} + O(\ge^{\abscissa-\sup f}),
  \qq \text{as } \ge \to 0^+.
\end{align}
Also observe that since $\sup f \leq \abscissa-\gd$, we have $\abscissa  - \sup f \geq \gd$, and hence
\linenopax
\begin{align}\label{eqn:order-est}
  O\left(\ge^{\abscissa-\sup f}\right) = O(\ge^{\gd}) = o(1),
  \qq\text{as } \ge \to 0^+.
\end{align}
To see that \tiling is Minkowski measurable with Minkowski content
\linenopax
\begin{align}
  c_\abscissa
  &= \res[s=\abscissa]{\gzL(s)} \frac{\spam_\abscissa(\gen)}{d-\abscissa},
    \label{eqn:nonlattice-Minkowski-content}
\end{align}
we reason as in the proof of \cite[Thm.~8.36]{FGCD}. We first show that the sum in \eqref{eqn:Minkowski-measurability-claim} is absolutely convergent and tends to $0$ as $\ge \to 0^+$; see Definition~\ref{def:Minkowski-content-and-dimension}. Indeed, note that \eqref{eqn:residue-bound} and \eqref{eqn:nonlattice-Minkowski-content-gen} implies
\linenopax
\begin{align*}%\label{eqn:}
  |c_\pole| \leq C \frac{|\spam_\pole(\gen)|}{|d-\pole|},
  \qq \text{for } \pole \in \DL(W),
\end{align*}
where the positive constant $C$ is as in \eqref{eqn:residue-bound}.
Therefore, %in view of \eqref{eqn:Fractal-tube-formula-with-error-term-simplified},
for a fixed $\ge > 0$, the sum in \eqref{eqn:Minkowski-measurability-claim} can be compared to
\linenopax
\begin{align}\label{eqn:polestrip-sum}
  \sum_{\pole \in \sD_\sL(W) \cap \polestrip_\gd} \frac1{|\pole|^{2}},
\end{align}
which converges by the density estimate (3.10) of \cite[Thm.~3.6]{FGCD}, according to which the poles of \gzL have a \emph{linear} density. In other words, \eqref{eqn:polestrip-sum} converges because $\sum_{n=1}^\iy \frac1{n^{2}} < \iy$.% for any $d \geq 1$.

This argument enables us to apply the method of proof of \cite[Thm.~5.17]{FGCD}
to deduce
\linenopax
\begin{align*}%\label{eqn:order-estimate-on-the-strip}
  \sum_{\pole \in \sD_\sL(W) \cap \polestrip_\gd} c_\pole \ge^{-\pole}
  = o(\ge^{-\abscissa}),
  \qq\text{as } \ge \to 0^+,
\end{align*}
in light of \eqref{eqn:polestrip}, and hence that
\linenopax
\begin{align}\label{eqn:order-estimate-on-the-strip}
  \sum_{\pole \in \DL(W) \cap \polestrip_\gd} c_\pole \ge^{d-\pole}
  = o(\ge^{d-\abscissa}) = o(1),
  \qq\text{as } \ge \to 0^+,
\end{align}
since $d > \abscissa$. Observe that the sum in \eqref{eqn:order-estimate-on-the-strip} converges for each fixed $\ge > 0$. Now one can see that \eqref{eqn:Minkowski-measurability-claim}, \eqref{eqn:order-est}, and \eqref{eqn:order-estimate-on-the-strip} imply that
\linenopax
\begin{align}\label{eqn:Minkowski-measurability-nonlattice-satisfied}
  \ge^{-(d-\abscissa)} V(T,\ge)
  = c_\abscissa + o(1),
  \qq \text{as } \ge \to 0^+.
\end{align}
  %and since $0 < c_\abscissa < \iy$, we conclude that \tiling is Minkowski measurable with Minkowski content equal to $c_\abscissa$.
  In light of \eqref{eqn:Minkowski-content}, it follows that
  \linenopax
  \begin{align}\label{eqn:Minkowski_D-content}
    \Mink_\abscissa(\tiling)
    = \lim_{\ge \to 0^+} \ge^{-(d-\abscissa)} V(T, \ge)
    = c_\abscissa,
  \end{align}
  which is both positive and finite, as noted above. %; see \eqref{eqn:nonlattice-Minkowski-content} and the remarks preceding it.
  %In fact, since $\Mink_\abscissa(\tiling)$ is a nonzero limit of positive numbers, it must be the case that $\Mink_\abscissa(\tiling) > 0$, whence \eqref{eqn:nonlattice-Minkowski-content} implies $\spam_\abscissa(\gen) > 0$. Indeed, by \cite[Rem.~5.6]{Pointwise} or \cite[Thm.~2.16]{FGCD}, the residue $\res[\abscissa]{\gzL(s)}$ is given by the first equality of \eqref{eqn:lattice-residues}, and thus $\res[\abscissa]{\gzL(s)} > 0$.
  Since $0 < \Mink_D(\tiling) < \iy$, it follows
  %from Remark~\ref{rem:D=alpha}
  that $\dim_M \negsp[4]\tiling = D$ and that \tiling is Minkowski measurable. This concludes the proof of Theorem~\ref{thm:Mink-meas} in the nonlattice case.
\end{proof}

\begin{remark}\label{rem:upper-and-lower}
  One can use the explicit form of \eqref{eqn:xG}
  %refine the argument leading up to \eqref{eqn:V-decay}
  in the proof of Theorem~\ref{thm:Mink-meas} to obtain more information on the periodic function \xG. In particular, the methods of \cite[\S8 and \S10]{FGCD} allow one to show that \xG is bounded away from 0 and from \iy, and hence that $0 < \Minkl(\tiling) < \Minku(\tiling) < \iy$ in the lattice case. %, the upper and lower Minkowski contents satisfy .
  %, and $D_\tiling=D$; that is, the Minkowski dimension of \tiling coincides with the abscissa of convergence of \gzL.
% More precisely, using \eqref{eqn:V-decay}, one first shows that \eqref{eqn:lower=upper-content} holds, except with $D_\tiling$ replaced by $D$ in the definition of $\Minkl(\tiling)$ and $\Minku(\tiling)$ provided in \eqref{eqn:upper-Minkowski-content}--\eqref{eqn:lower-Minkowski-content}. In light of Remark~\ref{rem:D=alpha}, this then implies that $D_\tiling=D$ and so \eqref{eqn:lower=upper-content} holds with $\Minkl(\tiling)$ and $\Minku(\tiling)$ now defined exactly as in \eqref{eqn:upper-Minkowski-content}--\eqref{eqn:lower-Minkowski-content}.
\end{remark}

\begin{remark}\label{rem:lattice-detail}
  Note that the proof of Theorem~\ref{thm:Mink-meas} in the lattice case is only given for a (monophase) tiling with a single generator. For the case of multiple generators, it is possible for cancellations to occur in the formula for \gzL, which results in the disappearance of some of the complex dimensions. Showing that \tiling is not Minkowski measurable in this case requires some care: one must check that after such cancellations, there still remains infinitely many complex dimensions on the line $\Re s = D$ (as is the case when $d=1$; cf. \cite[Thm.~8.25 and Cor.~8.27]{FGCD}). These technical issues are beyond the scope of the present paper but will be considered in \cite{Minko}, along with some possible counterexamples.
\end{remark}

\begin{remark}\label{rem:TFCD}
  The proof of Theorem~\ref{thm:Mink-meas} provides a detailed explanation of the argument behind \cite[Rem.~10.6]{TFCD}, while Theorem~\ref{thm:Mink-meas} itself justifies, completes, and strengthens \cite[Cor.~8.5]{TFCD}. While  \cite{TFCD} pertains to the \emph{distributional} context, the present (pointwise) result still applies (and is, in fact, stronger).
\end{remark}

\begin{comment} %obsolete with new Thm~\ref{thm:Mink-meas}

\begin{remark}[The exceptional case]
  \label{rem:exceptional-case}
  In the exceptional case when $\spam_{\abscissa+\ii m \per}(\gen) = 0$ for all $m \in \bZ \less\{0\}$, but still assuming that $\spam_{\abscissa}(\gen) \neq 0$, it follows from the above proof of Theorem~\ref{thm:Mink-meas} that the periodic function \xG of \eqref{eqn:xG} is constant and positive. Consequently, \tiling is always Minkowski measurable in this exceptional case, regardless of whether the scaling zeta function is of lattice or nonlattice type. Indeed, it must be that \tiling has Minkowski dimension \abscissa and Minkowski content given by \eqref{eqn:nonlattice-content}. An entirely analogous comment holds in regard to Theorem~\ref{thm:Minkowski-measurability-fractals} just below, mutatis mutandis. At this stage, we do not know of any example for which the exceptional case is realized.
\end{remark}
\end{comment}

\section{Minkowski measurability of self-similar fractals}
\label{sec:fractals-themselves}

  In \cite[Thm.~6.2]{GeometryOfSST}, precise conditions are given for when the (inner) parallel sets of the tiles in a self-similar \emph{tiling} can be used to decompose the parallel sets of the corresponding self-similar \emph{set}. 
  In this section, we study the Minkowski measurability of self-similar sets $F\subset\bR^d$ to which a self-similar tiling $\sT=\sT(O)$ with a monophase generator $G$ can be associated. The main geometric requirement needed to transfer the results obtained above for $\sT$ to the associated self-similar set $F$ is that $O$ can be chosen such that $\bd O\subset F$. 
Let us continue to write $K:=\cj{O}$ and $A_{-\ge}:=\{x\in A\suth d(x,A^c)\le \ge\}$ and $T := \bigcup_{R\in\sT} R$, and recall that $A_\ge$ (or $K_\ge$) is as defined in \eqref{eqn:A_ge}.

\begin{theorem}[Compatibility theorem {\cite[Thm.~6.2]{GeometryOfSST}}]\label{thm:compatibility}
  For the inner parallel set of an open set $A\subset\bR^d$, one has the disjoint decomposition
  \linenopax
  \begin{align}\label{eqn:disjoint-decomposition}
    \attr_\ge\setminus \attr = %\bigcup_{R\in\sT(O)} R_{-\ge} 
    T_{-\ge} \cup \left(K_\ge \setminus K\right),
    \qq\text{for all }\; \ge \geq 0,
  \end{align}
  if and only if the following compatibility condition is satisfied:
  \linenopax
  \begin{align}\label{eqn:compatibility-condition}
    \bd K \ci F.
  \end{align}
\end{theorem}

  In this case, a tube formula for the self-similar set \attr can be obtained simply by adding to $V(T,\ge)$ the (outer) parallel volume $\gl_d(K_\ge\setminus K)$ of $K$;
  %as in \eqref{eq:Steiner-formula-ext}
  %(although note that in the present context, $K$ need not be convex)
  see the examples in \cite[\S6]{Pointwise}. 

  Recall from \eqref{eqn:M-content} the definition of the ($\alpha$-dimensional) Minkowski content $\sM_\alpha(A)$ of a set $A\subset\bR^d$, and that $A$ is Minkowski measurable if and only if the number $\sM_\alpha(A)$ exists and is positive and finite. 
The relation \eqref{eqn:disjoint-decomposition} suggests that for the existence of the Minkowski content of $F$, it is not only the Minkowski content $\sM_D(\sT)$ of the tiling $\tiling$ which plays an important role. There is also a contribution
of the \emph{outer Minkowski content} $\bM_D(K)$ of the tiled set $K=\overline{O}$, which is defined as follows.

\begin{defn}\label{def:outer-Mink-content}
  Let $A\subset\bR^d$ and $\alpha\in[0,d]$. When the limit exists, the number
\linenopax
  \begin{align}\label{eqn:outer-Mink-content}
    \bM_\ga(A):=\lim_{\ge\to 0^+} \ge^{\ga-d} (\gl_d(A_\ge)-\gl_d(A))
  \end{align}
is called the \emph{outer \ga-dimensional Minkowski content} of $A$.  
\end{defn}
For sets $A$ with $\lambda_d(A)=0$ it obviously coincides with the usual Minkowski content.
(In general, $\bM_\alpha(A)$ is also equivalent to the relative Minkowski content of $A$ (or $\bd A$) relative to the set $A^c$, as discussed in \cite{Zubrinic2}, for example.)
Therefore, for the self-similar set $F$, it makes no difference whether we use the usual Minkowski content or its outer version. For the contribution of the set $K$, however, the outer Minkowski content is exactly the right notion.

Since Theorem~\ref{thm:Minkowski-measurability-fractals} requires the self-similar set \attr to satisfy the OSC and the hypotheses of Theorem~\ref{thm:Mink-meas}, it follows from this latter theorem (and Remark~\ref{rem:similarity-dimn}) that Theorem~\ref{thm:Minkowski-measurability-fractals} pertains to a situation where $D = \gd(\attr) = \dim_M \negsp[4]\attr = \dim_M \negsp[4]\tiling$.

\begin{remark}
  In the following theorem (Theorem~\ref{thm:Minkowski-measurability-fractals}), we assume that $\bM_\abscissa(K)$ exists with $0 \leq \bM_\abscissa(K) < \iy$. This is equivalent to the assumption that either $K$ is (outer) Minkowski measurable of dimension $D$ or that its $D$-dimensional outer Minkowski content vanishes; see \eqref{eqn:outer-Mink-content} and Remark~\ref{rem:fractal-boundaries}.
\end{remark}

\begin{theorem}[Minkowski measurability of self-similar fractals, monophase case]
  \label{thm:Minkowski-measurability-fractals}
%\marginpar{is the condition $\Gamma_D(G)\neq 0$ really needed in Theorem~\ref{thm:Mink-meas}?; maybe some formula for $\Mink(F)$ and some statement on the average Minkowski content should be added ...}
Let $F$ be a self-similar set in $\bR^d$ which has Minkowski dimension $D \in (d-1,d)$ and satisfies the open set condition. %
%\footnote{See Remark~\ref{rem:coincidental-abscissae}.}
Assume there exists a feasible open set $O$ for $F$ such that
\linenopax
\begin{align} \label{eqn:compatibility}
  \bd O \subset F,
\end{align}
and such that the associated tiling $\sT(O)$ has a single monophase generator, and assume that for the closure $K:=\overline{O}$ of $O$, the limit $\bM_\abscissa(K)$ %(from \eqref{eqn:outer-Mink-content}) 
exists and satisfies $0 \leq \bM_\abscissa(K) < \iy$. %(In other words, assume that either $K$ is (outer) Minkowski measurable of dimension $D$ or its $D$-dimensional outer Minkowski content vanishes; see \eqref{eqn:outer-Mink-content} and Remark~\ref{rem:fractal-boundaries}.)
%where $\bM_D(K)$ is the outer Minkowski content of $K$ defined as in \eqref{eqn:outer-Mink-content}.
%
%  Furthermore, assume that the generator $G$ ($:=K\setminus \bigcup_{n=1}^N \simt_n (O)$) of the associated self-similar tiling $\tiling=\tiling(O)$ is monophase with $\Gamma_D(G)\neq 0$%
 % \footnote{In which case the proof of Theorem~\ref{thm:Minkowski-measurability-fractals} clearly shows that $\spam_\abscissa(\gen) > 0$.}
%  and (in the lattice case) that $\spam_{\abscissa+\ii m \per}(\gen) \neq 0$ for some $m \in \bZ\less\{0\}$, where \per is the oscillatory period as in Remark~\ref{rem:dichotomy}.

Then %$F$ has Minkowski dimension \abscissa, and
$F$ is Minkowski measurable if and only if $F$ is nonlattice. In this case, the Minkowski content of $F$ is given by %the finite and positive number
\linenopax
\begin{align}\label{eqn:Mink-F}
  \Mink_\abscissa(F)=\Mink_\abscissa(\tiling)+\bM_\abscissa(K),
\end{align}
where %\tiling is the self-similar tiling associated with \attr (i.e., with the self-similar system $\{\simt_n\}_{n=1}^N$) and
$\Mink_\abscissa(\tiling)$ is the Minkowski content of \tiling, 
% given by the finite and positive expression \eqref{eqn:nonlattice-content}, 
and both $\Mink_\abscissa(F)$ and $\Mink_\abscissa(\tiling)$ also lie in the open interval $(0,\iy)$.
In the lattice case, the Minkowski content of $F$ does not exist, but the average Minkowski content $\overline{\sM}_D(F)$ exists and is given by
\linenopax
\begin{align}\label{eqn:avMink-F}
  \Minka_\abscissa(F)=\Minka_\abscissa(\tiling)+\bM_\abscissa(K),
\end{align}
where $\Minka_\abscissa(\tiling)$ is the average Minkowski content as in \eqref{eqn:lattice-content}. Again, $0<\Minka_\abscissa(F)<\iy$.
\end{theorem}

\begin{proof}
The assumption $D<d$ ensures that a self-similar tiling $\tiling(O)$ exists for each  feasible set $O$ for the open set condition of $F$; see Section~\ref{sec:Self-similar-tilings} or \cite[Thm.~5.7]{GeometryOfSST}. 
Now fix some $O$ such that the hypotheses on $O$ and $G$ are satisfied.
%the compatibility condition \eqref{eqn:compatibility} is satisfied, $\Mink(K)\in[0,\iy)$ and the generator of $\tiling(O)$ is monophase.
By \cite[Thm.~6.2]{GeometryOfSST}, \eqref{eqn:compatibility} ensures that we have the disjoint decomposition \eqref{eqn:disjoint-decomposition}.
Since $\lambda_d(F)=0$, this yields the relation
\linenopax
\begin{align}\label{eqn:V(F_e)}
  \gl_d(F_\ge) = V(T,\ge) + (\gl_d(K_{\ge})-\gl_d(K))
\end{align}
for the volume of these sets. Multiplying \eqref{eqn:V(F_e)} by $\ge^{D-d}$ and taking the limits as $\ge\to 0^+$, one obtains
\linenopax
\begin{align}\label{eqn:lim-eV(F_e)}
  \lim_{\ge \to 0^+} \ge^{D-d} \gl_d(F_\ge)
  = \lim_{\ge \to 0^+} \ge^{D-d} V(T,\ge)
    + \lim_{\ge \to 0^+} \ge^{D-d} (\gl_d(K_{\ge})-\gl_d(K)).
\end{align}
On the right side of \eqref{eqn:lim-eV(F_e)}, the second limit is $\bM_\abscissa(K)$, which exists in $[0,\iy)$ by assumption, while the first limit is $\Mink_\abscissa(\tiling)$, provided this number exists. By Theorem~\ref{thm:Mink-meas}, %which applies since the generator is monophase and satisfies $\spam_\abscissa(G)\neq 0$ and $\spam_{\abscissa+\ii m \per}(G)\neq 0$ for some $m \in \bZ\less\{0\}$, and since the abscissa of convergence of the associated scaling zeta function coincides with $D>d-1$,
this is the case exactly when the tiling (and thus $F$) is nonlattice. Hence the limit on the left side of \eqref{eqn:lim-eV(F_e)} (i.e., $\Mink_\abscissa(F)$) exists in (0,\iy) if and only if $F$ is nonlattice.
In particular, the set $F$ is not Minkowski measurable in the lattice case.
Formula \eqref{eqn:Mink-F} follows immediately from \eqref{eqn:lim-eV(F_e)} in the nonlattice case, while \eqref{eqn:avMink-F} follows similarly (in the lattice case) by comparing the corresponding average limits and noting that (as in the case $ \bM_D(K)$) an average limit exists whenever the corresponding limit exists, and then they both coincide.
%$\overline{\bM}_D(K) = \bM_D(K)$ (where $\overline{\bM}_D(K)$ is the average outer Minkowski content defined in the obvious way in analogy with Definition~\ref{def:Average-Minkowski-content}). 
%
%The formulas \eqref{eqn:Mink-F} and \eqref{eqn:avMink-F} are immediate consequences of the relation \eqref{eqn:Minkowski-content}, taking into account that $\Minka_\abscissa(K)=\Mink_\abscissa(K)$ in the latter case.
This concludes the proof of Theorem~\ref{thm:Minkowski-measurability-fractals}.
\end{proof}

\begin{remark}\label{rem:nonlattice=nonlattice}
  As was alluded to in the proof of Theorem~\ref{thm:Minkowski-measurability-fractals}, the self-similar fractal \attr is nonlattice if and only if the corresponding self-similar tiling \tiling is nonlattice. Note that, as discussed in Remark~\ref{rem:lattice-detail}, the proof (and statement) given here covers only the case of a single generator; for full details, see \cite{Minko}.
  
  The fact that a self-similar fractal \attr in \bRd (which satisfies the open set condition) is Minkowski measurable if and only if it is nonlattice was conjectured in \cite[Conj.~3, p.163]{Lap:Dundee}. Theorem~\ref{thm:Minkowski-measurability-fractals} resolves this conjecture (under the further conditions specified by the hypotheses). For $d=1$, this was established in \cite[Thm.~8.23 and Thm.~8.36]{FGCD}. We refer the interested reader to \cite[Rem.~8.17 and Rem.~8.39]{FGCD} for further discussion of this conjecture, and about earlier related work in \cite{LaPo1,Lap:Dundee,Falconer95} when $d=1$, and to \cite[\S12.5]{FGCD} for $d \geq 2$.
\end{remark}

\begin{remark}\label{rem:fractal-boundaries}
Note that \eqref{eqn:compatibility} implies $D\geq d-1$, since the (Minkowski) dimension of the boundary of a nonempty and bounded open subset of \bRd is at least $d-1$; see \cite{Lap:FD}. So the hypothesis $D>d-1$ just excludes the equality case $D=d-1$. This is necessary in order to apply Theorem~\ref{thm:Mink-meas}.

It is worth noting that one has $\bM_\abscissa(K)=0$, in particular, for all
feasible sets $O$ with finite surface area.
%for all feasible sets $O$ whose boundary has finite surface area.
Thus the corresponding condition in Theorem~\ref{thm:Minkowski-measurability-fractals} is not a restriction, provided $O$ can be chosen to have a nonfractal boundary.
Moreover, in case $\bM_\abscissa(K)=0$, the formulas \eqref{eqn:Mink-F} and \eqref{eqn:avMink-F} obviously simplify and the (average) Minkowski contents of the set $F$ and the associated tiling \tiling coincide.

In contrast, condition \eqref{eqn:compatibility} and the assumption that the generators are monophase impose serious restrictions on the class of sets covered by this result. To overcome the assumption of monophase generators, a suitable generalization of Theorem~\ref{thm:Mink-meas} is required which one might be able to derive from the general tube formulas obtained in \cite{Pointwise}; this issue will be examined in \cite{Minko}. Concerning the compatibility condition \eqref{eqn:compatibility}, there is a principal restriction on its validity. For certain sets, like Koch-type curves or totally disconnected sets, this condition is never satisfied; cf.~\cite[Prop.~6.3]{GeometryOfSST}. See also the top part of Figure~\ref{fig:examples} for a depiction of the self-similar tiling associated with the Koch curve.
 It was recently shown that a necessary and sufficient condition for the existence of a feasible set $O$ satisfying $\bd O\subset F$ is that the complement of $F$ has a bounded connected component, see \cite{PokW} and \cite[end of \S6]{GeometryOfSST}. This criterion can easily be checked and shows precisely the range of applicability and the limitations of the approach to study self-similar sets via self-similar tilings of feasible open sets.

Finally, we recall that as mentioned in Remark~\ref{rem:pluriphase}, it was recently shown in \cite{Kocak2} that any polytope in \bRd ($d \geq 1$) which admits an inscribed ball is monophase. Therefore, the monophase assumption in Theorem~\ref{thm:Mink-meas} and Theorem~\ref{thm:Minkowski-measurability-fractals} is satisfied under this condition on the generator. %if the generator of the corresponding self-similar tiling is a regular convex polyhedron. 
\end{remark}

\begin{exm}\label{exm:gasket}
  Consider the Sierpinski gasket \SG defined by the iterated function system %consisting of three contraction mappings $\{\simt_n\}_{n=1}^3$,  For concreteness, let
  \linenopax
  \begin{align}\label{eqn:SG-ifs}
      \simt_1(z) := \tfrac12 z, \q
    \simt_2(z) := \tfrac12 z + \tfrac12, \q \text{and}\q
    \simt_3(z) := \tfrac12 z + \tfrac14(1+\ii\sqrt3).
  \end{align}
  These mappings have contraction ratios $r_n=\frac12$, for $n=1,2,3$, and we have $D=D_F=D_\tiling=\gs = \log_{2}3$ because the system \eqref{eqn:SG-ifs} satisfies the open set condition. The scaling zeta function is
\linenopax
\begin{align}
  \gzL(s) = \frac1{1 - 3 \cdot 2^{-s}},
\end{align}
and the set of scaling complex dimensions is
\linenopax
\begin{align}
  \DL = \{D + \ii n\per \suth n \in \bZ\}
  \qq \text{for } D=\log_23, \; \per=\tfrac{2\gp}{\log2}.
\end{align}
  Consider the associated tiling $\tiling(O)$, where $O$ is the interior of the convex hull of \SG; this self-similar tiling is depicted in the middle part of Figure~\ref{fig:examples}.
Then $O$ is a feasible open set for \SG, and $K=\cj{O}$ satisfies the compatibility condition \eqref{eqn:compatibility}, along with the other assumptions of Theorem~\ref{thm:Minkowski-measurability-fractals}. In particular, this tiling has a single monophase generator $\gen = O \less \bigcup_{n=1}^3 \simt_n(K)$, which is an equilateral triangle with inradius $\genir = \tfrac1{4\sqrt3}$.
%This tiling was first considered in \cite[\S6.2]{Pe2}, but see also \cite{TFCD, SST, GeometryOfSST, TFSST, Pointwise}.
%See \cite[(9.12)]{TFCD} for a refined derivation of this formula; this self-similar tiling of the Sierpinski gasket is also discussed in \cite{GeometryOfSST, TFSST, Pointwise}.
%{eqn:Mink-F}. %{eqn:avg-Mink(Sierpinski)}. %.
%In light of \eqref{eqn:lattice-residues} and the above data, a straightforward computation yields the precise value of $\Minka_D(F)$. We leave this as an exercise for the reader.
  %
The tube formula for this tiling was computed in \cite[(6.29)]{Pe2} (see also \cite[\S9.3]{TFCD}) to be
\linenopax
\begin{align}
  V({\SG},\ge)
  &= \frac{\sqrt3}{16 \log2} \sum_{n \in \bZ}
    \left(-\frac1{D+\ii n\per} + \frac{2}{D-1+\ii n\per} - \frac{1}{D-2+\ii n\per}\right)\left(\frac{\ge}{\genir}\right)^{2-D-\ii n\per} \notag\\
  &\hstr[35]  + \frac{3^{3/2}}{2} \ge^2 - 3 \ge.
  \label{eqn:Sierpinski-tube-formula}
\end{align}
Note that this formula is exact, i.e., the sum is taken over all the complex dimensions and hence there is no error term as in Theorem~\ref{thm:Fractal-tube-formula-with-error-term}. See \cite[(9.12)]{TFCD} for a refined derivation of this formula; this self-similar tiling of the Sierpinski gasket is also discussed in \cite{GeometryOfSST, TFSST, Pointwise}.

Theorem~\ref{thm:Mink-meas} and Theorem~\ref{thm:Minkowski-measurability-fractals} imply that the Sierpinski gasket \attr is \emph{not} Minkowski measurable but that its average Minkowski content exists (in $(0,\infty)$) and is given by \eqref{eqn:lattice-content}. Applying \eqref{eqn:alpha-average-Minkowski-content} to \eqref{eqn:Sierpinski-tube-formula} directly would involve a certain amount of effort, but one can instead use Theorem~\ref{thm:Mink-meas} and Theorem~\ref{thm:Minkowski-measurability-fractals}. One can obtain the \nth[0] Fourier coefficient of \xG from \eqref{eqn:Sierpinski-tube-formula} by factoring $\ge^{2-D}$ out of the summation and extracting the term with $n=0$:
  \linenopax
  \begin{align}
    %\Minka_D(\tiling)
    \frac{\sqrt3}{16 \log2} \left(-\frac{1}{D}+\frac{2}{D-1}-\frac{1}{D-2}\right)\left(4\sqrt3\right)^{2-D}
    &= \frac{2 \sqrt3^{1-D}}{3D(D-1)(2-D)\log2}.
    \label{eqn:SG-tiling-content}
    %&= \frac{2^9 3^{\frac{\log3}{\log4}} (\log2)^3}{3^5\log(\frac{8}{3}) (\log3)^2+(\log2)^2 \log9}
  \end{align}
This is the mean value of \xG and hence the Minkowski content of \tiling; note that it is positive because $1<D<2$. Since Theorem~\ref{thm:Minkowski-measurability-fractals} applies and $\Minka_D(K)=0$ (because $D<2$), we conclude that the Sierpinski gasket \SG has average Minkowski content
  \linenopax
  \begin{align}\label{eqn:SG-content}
    \Minka_D(\SG) = \Minka_D(\tiling)
    = \frac{2 \sqrt3^{1-D}}{3D(D-1)(2-D)\log2}
    = 1.8125913503790578 \dots .
  \end{align}
%\begin{comment}\label{def:}
  This value can also be derived in a different way: substituting the scaling ratios into \eqref{eqn:lattice-residues} yields
  \linenopax
  \begin{align}\label{eqn:res(D,SG)}
    \res[\abscissa]{\gzL(s)}
    = \frac1{\sum_{n=1}^3 \cdot \left(\frac12\right)^D \log2}
    = \frac1{3 \cdot \frac13 \log2}
    = \frac1{\log2},
  \end{align}
  and the inner tube formula for the (monophase) generator is
  \linenopax
  \begin{align}\label{eqn:Sierpinski-generator-tube-formula}
    V(\gen,\ge) =
    \begin{cases}
    \tfrac12\ge - \sqrt3\ge^2,
      & 0 \leq \ge \leq \tfrac1{4\sqrt3}, \\
    \tfrac{\sqrt3}{16},
      & \ge \geq \tfrac1{4\sqrt3},
    \end{cases}
  \end{align}
  so $\crv_0 = -3^{1/2}$ and $\crv_1 = \frac12$. With this, \eqref{eqn:avg-Mink(tiling)} becomes
  \linenopax
  \begin{align}\label{eqn:res(D,SG)}
    \frac{1}{(2-D)\log2}\left(\frac{g^{D}}{D}(2)(-3^{3/2})+\frac{g^{D-1}}{D-1}(2-1)\left(\frac{3}{2}\right)\right)
    = \frac{2 \sqrt3^{1-D}}{3D(D-1)(2-D)\log2},
  \end{align}
  in agreement with \eqref{eqn:SG-content}.
%\end{comment}
\end{exm} 

An entirely analogous example could be provided for the Sierpinski carpet, whose associated self-similar tiling is depicted in the bottom part of Figure~\ref{fig:examples}.

\subsection*{Acknowledgements}
The authors are grateful for the careful and insightful comments provided by a diligent referee.

%\nocite{RatajWinter,Zubrinic,Zubrinic2}
\bibliographystyle{plain}%math}
\bibliography{strings}

\end{document}